\documentclass[11pt]{amsart}
\usepackage{cite,array}  
\usepackage[all]{xy}     
\usepackage{amssymb, amsmath,amsthm,color,amsfonts, mathrsfs}
\usepackage{subfiles}  
\usepackage[CJKbookmarks=true]{hyperref}
\usepackage[capitalize,nameinlink,noabbrev,nosort]{cleveref}

\newtheorem{thm}{Theorem}[section]
\newtheorem{prop}[thm]{Proposition}
\newtheorem{lemma}[thm]{Lemma}

\newtheorem{cor}[thm]{Corollary}
\newtheorem{defn}[thm]{Definition}

\newtheorem{remark}[thm]{Remark}

\newtheorem{qus}[thm]{Question}

 \numberwithin{equation}{section}

\newcommand{\dbb}[1]{[\![#1]\!]}
 \newcommand{\Fln}{\mathbb{F}\ell_{1, n-1; n}}

\newcommand{\bbC}{\mathbb C}
\newcommand{\bbP}{\mathbb P}
\newcommand{\bbQ}{\mathbb Q}
\newcommand{\bbZ}{\mathbb Z}

\newcommand{\ch}{\mathrm{ch}}
\newcommand{\qch}{\mathrm{qch}}
\newcommand{\Td}{\mathrm{Td}}

\begin{document}

\title{Towards small quantum Chern character}

\author{Hua-Zhong Ke}
\address{School of Mathematics, Sun Yat-sen University, Guangzhou 510275, P.R. China}
\email{kehuazh@mail.sysu.edu.cn}
\thanks{ 
 }
 
\author{Changzheng Li}
 \address{School of Mathematics, Sun Yat-sen University, Guangzhou 510275, P.R. China}
\email{lichangzh@mail.sysu.edu.cn}

\author{Jiayu Song}
 \address{School of Mathematics, Sun Yat-sen University, Guangzhou 510275, P.R. China}
\email{songjy29@mail2.sysu.edu.cn}


\thanks{}

\maketitle

\begin{abstract}
  We  show a quantum {version of} Chern character homomorphism  from the small quantum $K$-theory to the small quantum cohomology in the cases of projective spaces and  incidence varieties, whose classical limit gives the classical Chern character homomorphism. We also   provide a ring presentation of the small quantum $K$-theory of Milnor hypersurfaces.
\end{abstract}

\section{Introduction}
Let $X$ be a smooth complex projective variety. Let $K(X)=K^0(X;\mathbb{Q})$  be the Grothendieck  ring of  topological complex vector bundles on $X$ with rational coefficients. Consider the even part $H^{\rm ev}(X)$ of the rational cohomology $H^*(X)=H^*(X; \mathbb{Q})$. It is well-known that there is the Chern character ring homormorphism \cite{Gro, AtHi}
 $${\rm ch}:  K(X) \longrightarrow H^{\rm ev}(X).$$ 
There are quantum deformations of both sides. It is    natural to ask the following.

 \begin{qus}
 Is there a quantum Chern character  homomorhism   in the quantum world?
 \end{qus}

More precisely, the deformation $QH_{\rm big}(X)=(H^{\rm ev}(X)\otimes \bbQ\dbb{\mathbf{q}, \mathbf{s}}, *_\mathbf{s})$ of $(H^{\rm ev}(X), \cup)$, called the (even) big   quantum cohomology ring, encodes 
genus-zero  Gromov-Witten invariants of $X$ \cite{RuTi, BeFa, LiTi}.
Its $K$-analogue is the big quantum $K$-theory  $QK_{\rm big}(X)=(K(X)\otimes \bbQ\dbb{\mathbf{Q}, \mathbf{t}}, \star_\mathbf{t})$, encoding   
genus-zero  $K$-theoretic Gromov-Witten invariants \cite{Giv, Lee}.
Their restrictions to the fiber at the origin give the  small quantum cohomology/$K$-theory
 $$QH(X)=(H^{\rm ev}(X)\otimes\bbQ\dbb{\mathbf{q}}, *),\qquad QK(X)=(K(X)\otimes \bbQ\dbb{\mathbf{Q}}, \star)$$ respectively, which are relatively more accessible than the big ones. Whenever the small quantum product is finite, we can consider the corresponding polynomial version
  $$QH_{\rm poly}(X)=(H^{\rm ev}(X)\otimes\bbQ[\mathbf{q}], *),\qquad QK_{\rm poly}(X)=(K(X)\otimes \bbQ[\mathbf{Q}], \star)$$
 For instance, $QH_{\rm poly}(X)$ always makes sense for any Fano manifold $X$, while $QK_{\rm poly}(X)$ are known to make sense for very few examples. 
{There is a so-called {fake} quantum $K$-theory, which may be seen as building block of the big quantum $K$-theory. There has been the (big) quantum Chern(-Dold) character  \cite{Coa, Giv04, CoGi},  defined for the fake quantum $K$-theory. On the one hand, it gives  a  comparison  of the genus-0 fake $K$-theoretic  Gromov-Witten invariants with the cohomological ones in terms of symplectic geometry of loop spaces; moreover, by \cite[Remark 5.13]{IMT}, the big quantum $K$-theory is isomorphic to the big quantum cohomology as $F$-maniolds through the Hirzebruch-Riemann-Roch theorem for fake quantum $K$-theory \cite{GiTo}.  On the other hand, in general such isomorphism does not preserve the small loci, and it looks unclear how the (big) quantum Chern character is viewed as a quantum version of  the classical Chern character above.}
 
 The main aim of this paper is to explore a (small) quantum Chern character  ring homomorphism ${\rm qch}$ for the small quantum $K$-theory, such that  {its classical limit gives $\ch$,} i.e. the following diagram of ring homomorphisms is commutative, where vertical maps  are    natural projections.  
   \begin{equation} \label{commdiag}\xymatrix{
   QK(X) \ar@{->>}[d]_{\mod \mathbf{Q}}  \ar[r]^{\rm qch}
                & QH^*(X) \ar@{->>}[d]^{{}\!\!\!\!\mod \mathbf{q}}   \\
 QK(X)/(\mathbf{Q})= K(X)  \ar[r]^{\!\!\rm ch}
                & H^*(X)=QH^*(X)/(\mathbf{q})            }
\end{equation}
 As we will see below, we succeed to find the small quantum Chern character morphism $\mbox{qch}$  in the cases of projective spaces $\bbP^n$ and incidence varieties $\Fln$. 
 
 Both $\bbP^n$ and   $\Fln$  are examples of flag varieties $G/P$, which are Fano manifolds and are  central objects of study in enumerative geometry. There have been extensive studies of  $QH_{\rm poly}(G/P)$ (see e.g. the survey \cite{LeLi} and the references therein).
 There have also been studies of $QK(G/P)$ in individual cases \cite{BuMi, BCMP18,  BCP, LNS, KLNS, Xu, LLSY} from the lens of  Schubert calculus, giving the quantum multiplication of certain Schubert classes.
 Ring presentations have recently been provided for the small quantum $K$-theory of type $A$ partial flag varieities   \cite{KPSZ, GMSZ,  GM+23, GM+24, HK, MNS25a, MNS25b, AHK+25} and of type $C$ complete flag varieties \cite{KoNa}, as well as  for  general $G/P$ in terms of  $K$-theoretic version of Peterson’s ‘quantum = affine’ statement \cite{LLMS, IIM} proved   in the groundbreaking works \cite{Kat18,Kat19} (see also \cite{ChLe}).
 Moreover, the small quantum $K$ product is in fact finite for general $G/P$ \cite{ACT} (see also \cite{BCMP13, BCMP16} for certain Grassmannians), so that   $QK_{\rm poly}(G/P)$ is meaningful.

Let us start with the simplest case $\bbP^n$. The line bundle class $L:=[\mathcal{O}_{\bbP^n}(1)]$ generates $K(\mathbb{P}^{n})$ with  inverse $L^{-1}=[\mathcal{O}_{\bbP^n}(-1)]$, and the hyperplane class $h:=c_1(\mathcal{O}_{\bbP^n}(1))$ (or simply $c_1(L)$) generates $H^{*}(\mathbb{P}^{n})=H^{\rm ev}(\bbP^n)$.  We have $$QK_{\rm poly}(\mathbb{P}^{n})=\mathbb{Q}[L^{-1}, Q]/((1-L^{-1})^{n+1}-Q),\qquad  {QH}_{\rm poly}(\mathbb{P}^{n})=\mathbb{Q}[h, q]/(h^{n+1}-q),$$ 
both of which are isomorphic to $\bbQ[x]$ as $\bbQ$-algebras. However, it is less obvious whether there is an (iso)morphism compatible with  $\mbox{ch}$ in the sense of making diagram \eqref{commdiag} commutative.
To achieve our aim, we  need    the completion ${QH}(\bbP^n)=QH_{\rm poly}(\bbP^n)\otimes_{\bbQ[q]}\bbQ\dbb{q}$ and introduce the \textit{quantum Todd class}  of the tangent bundle $T_{\mathbb{P}^{n}}$,
 \begin{align}
  {\rm Td}_{q}(T_{\mathbb{P}^{n}}):=(\frac{h}{1-e^{-h}})^{n+1}\in {QH}(\bbP^n),
 \end{align}
in whose series expansion each term $h^k$ is read off as the quantum product $h^{*k}$. In particular, we have ${\rm Td}_q(T_{\bbP^n})\equiv {\rm Td}(T_{\bbP^n}) \mod \mathbf{q}$.
  As the warm-up theorem, we have 
 \begin{thm}[\textbf{\cref{thm: Pn}}]\label{mainthmPn}
The  map {\upshape$ \mbox{qch}: QK(\mathbb{P}^{n})\longrightarrow  {QH}(\mathbb{P}^{n})$},  given by 
$$ L^{-1} \mapsto e^{-h}=\sum_{m=0}^{+\infty}\frac{(-h)^{*m}}{m!},\qquad Q \mapsto q~({\rm Td}_{q}(T_{\mathbb{P}^{n}}))^{-1}, $$
is a well-defined ring homomorphism, and its classical limit gives {\upshape $\mbox{ch}$}.
\end{thm}
\begin{remark}
 As pointed out in \cite[Remark 5.13]{IMT},  Givental-Tonita's work \cite{GiTo} implies that the big quantum $K$-theory is isomorphic to the big quantum cohomology as $F$-maniolds. From the discussion with Hiroshi Iritani,   such isomorphism does not restrict to the small locus in the following sense: when $n\ge2$,  there is no  continuous ring homomorphism ${\rm ch}^{q}: QK(\mathbb{P}^{n}) \rightarrow QH(\mathbb{P}^{n})$ (with respect to the natural $Q$-adic/$q$-adic topology) such that   ${\rm ch}^{q}(Q)\in \bbQ\dbb{q}\setminus\{0\}$ and \eqref{commdiag} is a commutative diagram of ring homomorphisms.  In \cref{mainthmPn}, the image of $Q$ does lie in  $QH(\bbP^n)\setminus \bbQ\dbb{q}$.  
\end{remark}
 The second simplest case may be the incidence varieties $\Fln=\{V_1\leq V_{n-1}\leq \bbC^n\mid \dim V_1=1, \dim V_{n-1}=n-1\}$.
Let $\underline{\bbC^n}$ denote the trivial $\bbC^n$-bundle over $\Fln$, and $\mathcal{S}_1$ 
(resp. $\mathcal{S}_{n-1}$) denote the tautological vector subbundle  whose fiber at a point $V_\bullet$ is just the vector space $V_1$ (resp. $V_{n-1}$).
Let $L_1$ (resp. $L_2$) denote the class of the dual line bundle $\mathcal{S}_1^\vee$ (resp. quotient line bundle $\underline{\bbC^n}/\mathcal{S}_{n-1}$) in $K(\Fln)$. We have the inverse  $L_1^{-1}=[\mathcal{S}_1]$ and $L_2^{-1}= [(\underline{\bbC^n}/\mathcal{S}_{n-1})^\vee]$ in $K(\Fln)$, and denote $h_i:=c_1(L_i)\in H^2(\Fln)$ for $i=1, 2$.
We use the ring presentation for $QH_{\rm poly}(\Fln)$ with generators $\{h_1, h_2\}$ in \cite[Proposition 7.2]{ChPe} (see \cref{prop: QHFl} below). For $QK(\Fln)$, the ring presentation could be obtained by simplifying  the Whitney presentation in \cite[Theorem 1.3]{GM+23} (with $2n-1$ generators) or derived from the quantum Littlewood-Richardson rule in \cite{Xu}. Here we consider the ring presentation with generators $\{L_1^{-1}, L^{-1}_2\}$ read off from \cref{thm: presentationQK} as a special case of Milnor hypersurfaces. They are of the following form 
$$QK(\Fln)={\bbQ[L_1^{-1},L_2^{-1}]\dbb{Q_1, Q_2}\over (F_1^Q, F_2^Q)}, \qquad  {QH}(\Fln)={\bbQ[h_1,h_2]\dbb{q_1, q_2}\over (f_1^q, f_2^q)}$$
with $F_i^Q=F_i^Q(L_1^{-1}, L_2^{-1}, Q_1, Q_2)$ and $f_i^q=f_i^q(h_1, h_2, q_1, q_2)$ being polynomials. Similar to $\bbP^n$, we introduce the quantum Todd class  
\begin{align}\label{ToddFl}
     ( {\rm Td}_{q}(L_i^{\oplus n}))^{-1} * {\rm Td}_{q}(L_1 \otimes L_2)   := (\frac{1-e^{-h_i}}{h_i})^{n}/(\frac{1-e^{-h_1-h_2}}{h_1+h_2}).
\end{align}  
In particular,    $( {\rm Td}_{q}(L_1^{\oplus n})) * {\rm Td}_{q}(L_1 \otimes L_2)^{-1}   \mod \mathbf{q}$ is the classical Todd class of the difference between the pullback of $T_{\mathbb{F}\ell_{1; n}}$ and the line bundle $L_1 \otimes L_2$ (see \cref{subsec: QchFl} for more details).
We have the following quantum Chern character homomorphism for $\Fln$.
 \begin{thm}[\textbf{\cref{prop--qchwelldefineFl}}]\label{mainthmFln}
The  map {\upshape$ \mbox{qch}: QK(\Fln)\longrightarrow  {QH}(\Fln)$},  given by 
 \begin{align*}
  \hspace{1em} L_{i}^{-1} &\mapsto e^{-c_1(L_i)}:=\sum_{m=0}^{+\infty}\frac{(-h_i)^{*m}}{m!},  \\
 Q_i &\mapsto q_i ( {\rm Td}_{q}(L_i^{\oplus n}))^{-1} * {\rm Td}_{q}(L_1 \otimes L_2),\qquad i=1,2.
 \end{align*}
 is a well-defined ring homomorphism, and its classical limit gives {\upshape $\mbox{ch}$}.
\end{thm}
 \noindent We notice that the inverse of (classical) Todd class  appeared in the Grothendieck-Riemann–Roch formula (see e.g. \cite[Theorem 14.5]{EiHa}).
 The inverse of Todd function also appeared in the study of BCOV theory \cite[Formula (38)]{CoLi}.

The degree-$(1, 1)$ hypersurfaces in $\bbP^{n-1}\times \bbP^{m-1}$, known as Milnor hypersurfaces, 
 are an important class of algebraic varieties  in topology and algebraic geometry.
Amongest them the smooth one is usually denoted as $H_{n-1, m-1}$. In particular, we have $H_{n-1, n-1}=\Fln$. We can view $H_{n-1, m-1}$ as a smooth Schubert variety of $\Fln$ via a natural embedding $\iota$ where $n\geq m\geq 2$ (see \cref{sub: Milnor} for more details). Then we have the line bundle classes $[\iota^*(\mathcal{S}_1^\vee)]$ and $[\iota^{*}(\underline{\bbC^n}/\mathcal{S}_{n-1})]$ in $K(H_{n-1, m-1})$, denoted again as $L_1$ and $L_2$ respectively by abuse of notation. 
By \cref{Classical K}, we have $K(H_{n-1, m-1})=\mathbb{Q}[L_1^{-1}, L_2^{-1}]/(F_1(L_1^{-1}, L_2^{-1}), F_2(L_1^{-1},L_2^{-1}))$,
  where $$F_1(x,y):=(1-y)^{m},\qquad F_2(x,y):=(-1)^{n-1}(1-x)^{n-1}+\sum_{t=1}^{m-1}(-1)^{n-1-t}x^{t-1}(1-x)^{n-t-1}(1-y)^{t}.$$ 
As we will show in \textbf{\cref{quantum K ring presentation}}, we have the following ring presentation. 
\begin{thm} \label{thm: presentationQK}
    Let $n\geq m\geq 3$. The small quantum $K$-theory  of    $H_{n-1, m-1}$  is presented by
    \begin{align*}
     QK(H_{n-1, m-1}) \cong \mathbb{Q}[L_1^{-1},L_2^{-1}]\dbb{Q_1,Q_2}/(F_1^{Q}, F_2^{Q}),
\end{align*}
 where  $F_1^{Q}(L_1^{-1},L_2^{-1},Q_1,Q_2)=F_1(L_1^{-1},L_2^{-1})-Q_2+Q_2 L_1^{-1}L_2^{-1}$ and
 $$F_2^{Q}(L_1^{-1},L_2^{-1},Q_1,Q_2) =F_2(L_1^{-1},L_2^{-1})-(-1)^{n-m}Q_2 (L_1^{-1})^{m-1}(1-L_1^{-1})^{n-m}-(-1)^{n-1}Q_1 L_2^{-1}.
 $$  
\end{thm}
\noindent 
On the one hand, there is an approach, known to experts, to derive a ring presentation  of  {the quantum $K$-theory of a Fano manifold $X$}. Namely, firstly one may compute the $K$-theoretic $J$-function $J_X$ of $X$; secondly, one may   find  the symbols of finite difference operators annihilating  $J_X$, which would give  relations in $QK(X)$ by  \cite[Proposition 2.10]{IMT} and \cite[Theorem 4.12]{HK24a}; thirdly, one may show those relations are sufficient by a Nakayama-type result in \cite[Appendix A]{GM+23}.  {We notice that when $K(X)$ is generated by line bundle classes, there has also been a heuristic argument on  obtaining the ring presentation of $QK(X)$ by using finite difference equations  in \cite{Lee1999}.} On the other hand, there are very few examples that could really be carried out  in this way,  including the type $A$ complete flag variety \cite[Appendix A]{AHK+25}. 
{Our \cref{thm: presentationQK} provides one new class of examples with this approach. We directly write down the $J$-function of $H_{n-1,m-1}$ by  the quantum Lefschetz principle \cite{Giv15, Giv20}.}    We notice that the  $K$-theoretic $J$-function has been calculated for type $A$ flag varieties \cite{GiLe, Tai} (see \cite{BrFi} for the quasi-map version for complete flag varieties).

\begin{remark}
 \cref{mainthmFln} can be generalized to Milnor hypersurfaces  $H_{n-1, m-1}$ as follows. 
 We first derive  the following ring presentation 
  $$QH(H_{n-1, m-1})\cong \bbQ[h_1, h_2]\dbb{q_1, q_2}/(h_2^m-q_2(h_1+h_2),  \sum_{k=0}^{m-1}(-1)^kh_1^{n-1-k}h_2^k-q_1-(-1)^{m-1}q_2h_1^{n-m}) $$
 by investigating  the quantum differential equations for Givental's (cohomological) $J$-function of $H_{n-1, m-1}$. Then we 
can  define the  map $\qch: QK(H_{n-1,m-1}) \rightarrow QH(H_{n-1,m-1}) $  by
  sending $L_{i}^{-1}$ to $e^{-h_{i}}$ and sending $Q_1$ (resp. $Q_2$) to $q_1 ( {\rm Td}_{q}(L_1^{\oplus n}))^{-1} * {\rm Td}_{q}(L_1 \otimes L_2)$  (resp. $q_2 ( {\rm Td}_{q}(L_2^{\oplus m}))^{-1} * {\rm Td}_{q}(L_1 \otimes L_2)$).  The strategy to show that  such ${\rm qch}$ gives a ring homomorphism  $\qch$ is similar to the case for $\Fln=H_{n-1, n-1}$. The full details will be slightly more involved, and we plan to address 
  the cohomological part in the study of $D$-module mirror symmetry for $H_{n-1, m-1}$ elsewhere. 
\end{remark}
The paper is organized as follows. In Section 2, we review the background of quantum cohomology and quantum $K$-theory. In Sections 3 and 4, we show the small quantum Chern character homomorphism for the cases $\bbP^n$ and $\Fln$ respectively. In Section 5, we derive a ring presentation of the small quantum $K$-theory of Milnor hypersurface $H_{n-1, m-1}$. Finally, we prove \cref{lemma--h1+h2nonzerodivisorQHFl} by studying mirror symmetry for $\Fln$.

\subsection*{Acknowledgements}

 The authors would like to thank Sergey Galkin, Alexander Givental, Tao Huang,  Hiroshi Iritani, Bai-Ling Wang and Rui Xiong  for  helpful discussions.
 The authors are  supported in part by the National Key Research and Development Program of China No.~2023YFA1009801.
H.~Ke is also supported in part by NSFC Grant 12271532. C. ~Li   is also supported in part by NSFC Grant 12271529.

\section{Preliminaries}

In the section, we review some basic facts on quantum cohomology and quantum $K$-theory, mainly following \cite{FuPa, Giv, Lee, IMT}.
For simplicity, we assume the topological $K$-theory of $X$ coincides with the algebraic $K$-theory of $X$, which is sufficient in the rest of this paper. We refer to \cite{Lee} for more details on the set-up of the quantum $K$-theory in general cases. 

\subsection{Cohomological/$K$-theoretic  Gromov-Witten invariants}
 Let  $\overline{\mathcal{M}}_{0,k}(X,\beta)$ denote the   moduli (Degline-Mumford) stack of  genus $0$, $k$-pointed  stable maps to $X$ of degree   $\beta\in H_{2}(X;\mathbb{Z})$. Let ${\rm ev}_{i}:\overline{\mathcal{M}}_{0,k}(X,\beta)\longrightarrow X $  be  the $i$-th evaluation map, and  $\mathcal{L}_i$ denote the universal cotangent line bundle on $\overline{\mathcal{M}}_{0,k}(X,\beta)$, where $1\leq i \leq k$. 
 Given nonnegative integers $d_i$ and classes $\xi_i\in H^{\rm ev}(X)$/$\gamma_i\in K(X)$, the cohomological/$K$-theoretic Gromov-Witten invariants  are respectively defined by  
\begin{align*}
    \langle \tau_{d_1}\xi_1, \tau_{d_2}\xi_2, \cdots,  \tau_{d_k}\xi_k  \rangle_{0,k, \beta}^{H}&=\int_{[\overline{\mathcal{M}}_{0,k}(X,\beta)]^{vir}} \prod_{i=1}^{k}(c_1(\mathcal{L}_i)^{d_i} \cup {\rm ev}_{i}^{*}(\xi_{i})),\\
     \langle \tau_{d_1}\gamma_1, \tau_{d_2}\gamma_2, \cdots,  \tau_{d_k}\gamma_k  \rangle_{0,k, \beta}&= \chi \bigg(\overline{\mathcal{M}}_{0,k}(X,\beta), \mathcal{O}^{vir}\otimes(\otimes_{i=1}^{k} \mathcal{L}_i^{\otimes d_i  }{\rm ev}_{i}^{*}(\gamma_i) )\bigg).
\end{align*}
Here   $[\overline{\mathcal{M}}_{0,k}(X,\beta)]^{vir}$ denotes the virtual fundamental class \cite{RuTi, BeFa,LiTi,FuOn}, which is of expected dimension $\dim X+k-3+\int_{\beta}c_1(X)$;   $\mathcal{O}^{vir}$  denotes  the virtual structure sheaf  \cite{Lee,BeFa}, and  { $\chi$ is the pushforward to a point $\mathrm{Spec}(\mathbb{C})$}.

Let $\overline{\mbox{NE}}(X)$ denote the Mori cone of effective curve classes. It follows from the definition that  cohomological/$K$-theoretic Gromov-Witten invariants of degree $\beta$ are vanishing unless $\beta\in\overline{\mbox{NE}}(X)$. 
We may  take  nef line bundles $L_i$'s, such that $\{h_i:=c_1(L_i)\}_{i=1}^r$ form a nef integral basis  of $H^{2}(X;\mathbb{Z})$/torison. We then introduce 
the Novikov variables $\{q_i\}_{i=1}^r$ (resp. $\{Q_i\}_{i=1}^r$) for quantum cohomology (resp. $K$-theory). For $\beta \in\overline{\mbox{NE}}(X)$, we have
 $d_i:=d_i(\beta)=\int_\beta h_i\geq 0$, and denote 
$$q^{\beta}:=\prod_{i=1}^{r}q_{i}^{d_i},\qquad Q^{\beta}:=\prod_{i=1}^{r}Q_{i}^{d_i}.$$  
 \subsection{Quantum cohomology }\label{subsec: QH}
 Let $\{\xi_i\}_{i=1}^N$ be a basis of $H^{\rm ev}(X)$, and $\{\xi^j\}_j$ be the dual basis of $H^{\rm ev}(X)$, where   $\int_{[X]} \xi_{b} \cup \xi^{a} =\delta_{a,b}$. Write $\mathbf{s}=\sum_{i=1}^Ns_{i}\xi_i \in H^{\rm ev}(X)$. The (even) big quantum cohomology  $QH_{\rm big}(X)=(H^{\rm ev}(X)\otimes \bbQ\dbb{\mathbf{q}, \mathbf{s}}, *_{\mathbf{s}})$ is a $\bbQ\dbb{\mathbf{s}, \mathbf{q}}$-algebra with the big quantum product defined by 
 \begin{align} \label{quantumcoh}
     \xi_{i} \ast_{\mathbf{s}} \xi_{j}:=\sum_{k=0}^{\infty}\sum_{\beta \in {\rm \overline{NE}}(X)} \sum_{a=1}^{N}\frac{q^{\beta}}{k!}  \langle\xi_i,\xi_j,\xi^{a},\mathbf{s},\cdots,\mathbf{s}\rangle_{0, k+3, \beta}^H \xi_{a}.
 \end{align}
  The restriction of $QH_{\rm big}(X)$ to $\mathbf{s}=0$, i.e. $\xi_i\ast\xi_j= \sum_{\beta, a}   \langle\xi_i,\xi_j,\xi^{a}\rangle_{0, 3, \beta}^H \xi_{a}{q^{\beta}}$, gives the small quantum cohomology  $QH(X)=(H^{\rm ev}(X)\otimes \bbQ\dbb{\mathbf{q}}, *)$.
  
For $\alpha\in H^{\rm ev}(X)$ and $f(x)=x^m$ where $m\geq 0$, we take the notation conventions 
$$f_{\rm cl}(\alpha)=\alpha^{\cup m}:=\underbrace{\alpha\cup\cdots \cup \alpha}_{m \,\,\rm \mbox{copies}},\qquad f(\alpha)=\alpha^{* m}:=\underbrace{\alpha*\cdots * \alpha}_{m \,\,\rm \mbox{copies}}.$$
We can take  similar conventions for formal power series in multi-variables by observing the following. For any $f(x_1, \ldots, x_k)\in \bbQ\dbb{x_1, \ldots, x_k}$ and any homogeneous $\alpha_i\in H^{\rm ev}(X)\setminus H^0(X)$, $1\leq i\leq k$, we have
\begin{enumerate} \label{clandqu}
    \item $f_{\mathrm{cl}}(\alpha_1,..,\alpha_k)\in H^{\rm ev}(X)$, by the nilpotency of $\alpha_i$'s in $H^{\rm ev}(X)$; 
     \item $f(\alpha_1,..,\alpha_k)\in QH(X)$ for the reason:   for each $q^\beta$, by the dimension axiom,   there are only finitely many $k$-tuples $(m_1,...,m_k)$ such that $q^\beta$ appears in $\alpha_1^{\ast m_1}\ast\cdots\ast\alpha_k^{\ast m_k}$.
\end{enumerate}
Moreover, under   the natural surjective ring  homomorphism $QH(X)\xrightarrow{\textrm{mod }q}H^{\rm ev}(X)$, we have $f(\alpha_1,...,\alpha_k)\mapsto f_{\mathrm{cl}}(\alpha_1,...,\alpha_k)$.

More generally, for  $\alpha\in H^2(X)$ and function   $f(x)$   holomorphic at $x=0$,   we can define $f_{\mathrm{cl}}(\alpha)\in H^{\rm ev}(X)$ and $f(\alpha)\in QH(X)$ by the Taylor expansion of $f(x)$ at $x=0$. In this article, we    consider the cases $f(x)=e^x$, $\frac{1-e^{-x}}{x}$ or $\frac{x}{1-e^{-x}}$,   giving $f(\alpha)\in QH(X)$.

\subsection{Quantum $K$-theory}
The definition of quantum $K$-theory is more involved. Let  
      $\mathbf{t}=\sum_{i=1}^Nt_{i}\gamma_{i}\in K(X)$. We first introduce  the quantum $K$-potential, which is  a generating series of genus-zero $K$-theoretic Gromov-Witten invariants given by  
$$\mathcal{G}(\mathbf{t}, \mathbf{Q}):=\frac{1}{2}\sum_{i, j}t_it_j g_{ij}+\sum_{k=0}^{\infty}\sum_{\beta \in {\rm \overline{NE}}(X)}\frac{Q^{\beta}}{k!} \langle \mathbf{t},  \cdots,  \mathbf{t}  \rangle_{0,k, \beta}$$
with $g_{ij}=g(\gamma_i, \gamma_j):=\chi(X, \gamma_i\otimes\gamma_j)$ being the classical metric on $K(X)$ given by the Euler characteristic map. 
 
 Denote $\partial_i:=\frac{\partial}{\partial t_i}$.
  The quantum $K$-metric is defined by 
$$G_{ij}=G(\gamma_i,\gamma_j):=\partial_{i}\partial_{j}\mathcal{G}.$$
We have $G_{ij}\equiv g_{ij} \mod \mathbf{Q}$, and $\big(G_{ij}\big)\in GL_N(\bbQ\dbb{\mathbf{Q}, \mathbf{t}})$.
\begin{defn}
 The big quantum $K$-theory $QK_{big}(X):=(K(X)\otimes \dbb{\mathbf{Q}, \mathbf{t}},\star_{\mathbf{t}})$ is a $\dbb{\mathbf{Q}, \mathbf{t}}$-algebra with the big quantum $K$-product $\gamma_i \star_{\mathbf{t}} \gamma_j$ defined by using the quantum $K$-metric, 
$$G(\gamma_i \star_{\mathbf{t}} \gamma_j,\gamma_k):=\partial_i \partial_j \partial_k \mathcal{G}.$$
\end{defn}

 The big quantum $K$-theory $QK_{\mathrm{big}}(X)$ is a commutative and associative algebra with identity $\mathbf{1}$ (lying in $K(X)$).  Its restriction to $\mathbf{t}=0$ gives the small quantum $K$-theory 
 $QK(X)=(K(X)\otimes \mathbb{Q}\dbb{\mathbf{Q}},\star)$.

\subsection{Quantum connection and $K$-theoretic $J$-function}
Let $\hbar$ be a formal variable. The quantum connection is given by the operators
$$\nabla_{i}^{\hbar}:={(1-\hbar)}\frac{\partial}{\partial t_{i}}-\gamma_{i} \star_{\mathbf{t}},  ~~~1 \leq i\leq N,
$$
acting on $K(X)\otimes \mathbb{Q}[\hbar,\hbar^{-1}]\dbb{\mathbf{Q},\mathbf{t}}$.  It  could also be seen as a connection on the tangent bundle $T_{K(X)}$ by identifying $\gamma_{i}$ with $\frac{\partial}{\partial t_{i}}$. This quantum  connection is flat, with   the fundamental solution   $\{S_{ab}\}_{1 \leq a,b \leq N }$  to the $K$-theoretic quantum differential equation $\nabla_{i}^{\hbar}(\sum_{b=1}^{N}S_{ab}\gamma_{b})=0$ given by

$$ S_{ab}(\mathbf{t},\mathbf{Q})=g_{ab}+\sum_{l=0}^{\infty}\sum_{d \in {\rm \overline{NE}(X)}}\frac{Q^{d}}{l!}\langle \gamma_a, \mathbf{t}, \cdots,  \mathbf{t},\frac{\gamma_b}{1-\hbar\mathcal{L}} \rangle_{0,l+2, d}.$$
Here  $\frac{\gamma_b}{1-\hbar\mathcal{L}}=\sum_{n=0}^{\infty}\gamma_b \hbar^{n} \mathcal{L}^{\otimes n}$,  where $\mathcal{L}$ is the universal cotangent line bundle at the $(l+2)$th-marked point on  $\overline{\mathcal{M}}_{0,l+2}(X, d)$. {Following \cite{IMT}, we introduce an endomorphism-valued function $T\in End(K(X))\otimes \mathbb{Q}(\hbar)\dbb{Q,t}$ by the formula $g(T\gamma_a,\gamma_b)=S_{ab}$.}

\begin{defn}
    The  big $K$-theoretic $J$-function  \footnote{The $K$-theoretic $J$-function in \cite[Definition 2.4]{IMT} is equal to our $(1-\hbar)J$ with $\hbar=q$.} of $X$
    is defined by 
    \begin{align*}
        J_X(\hbar, \mathbf{Q}, \mathbf{t})&:=T(\mathbf{1})
        =\mathbf{1} +\sum_{l=0}^{\infty}\sum_{d \in {\rm \overline{NE}(X)}}{\sum_{i,j}}\frac{Q^{d}}{l!}\langle \mathbf{1}, \mathbf{t}, \cdots,  \mathbf{t},\frac{\gamma_i}{1-\hbar \mathcal{L}} \rangle_{0,l+2, d} ~g^{ij} \gamma_{j}.
    \end{align*}
   The restriction $J_X(\hbar, \mathbf{Q})=J_X(\hbar, \mathbf{Q},\mathbf{t})|_{\mathbf{t}=0}$ is called the small $K$-theoretic $J$-function.
\end{defn}
 Givental and Tonita \cite{GiTo} showed that the space $T(K(X) \otimes \mathbb{Q}[\hbar,\hbar^{-1}]\dbb{\mathbf{Q},\mathbf{t}})$ is preserved by the operator $L_i^{-1}\hbar^{Q_i \partial_{Q_i}}$ for all $1 \leq i \leq r$, 
where  {$L_i^{-1}$ acts by tensor product, and} $\hbar^{Q_{i}\partial_{Q_i}}$ is the $\hbar$-shift operator acting on functions in $Q_1,Q_2,\cdots,Q_r$ by
    $$f(Q_1,Q_2,\cdots,Q_r) \mapsto f(Q_1,\cdots,Q_{i-1},\hbar Q_{i},Q_{i+1},\cdots, Q_r).$$
 {Set $A_{i}:=\left(T^{-1}\circ L_{i}^{-1}\hbar^{Q_i\partial_{Q_i}} \circ T \right)|_{\mathrm{t}=0,\hbar=1}$, which is an element in $\mathrm{End}(K(X)) \otimes \mathbb{Q}\dbb{\mathbf{Q}}$. Suppose that a $\hbar$-difference operator $D\in \mathbb{Q}[\hbar,\mathbf{Q}]\left \langle L_i^{-1}\hbar^{Q_i\partial_{Q_i}} \right \rangle$ satisfies 
 \begin{align} \label{D-module relation}
          D(L_1^{-1}\hbar^{Q_1 \partial_{Q_1}},\cdots,L_r^{-1}\hbar^{Q_r\partial_{Q_r}},\hbar, Q_1, \cdots, Q_r) J_{X}(\hbar, \mathbf{Q})=0.
      \end{align}
Then $D(A_{1},\cdots, A_{r},1, Q_1, \cdots, Q_r)(\mathbf{1})=0$ by the definition of $A_i$'s. To simplify $A_i$, we note that if the coefficient of $Q^d$ in $L_i^{-1}\hbar^{Q_i\partial_{Q_i}}(J_{X}(\hbar,\mathbf{Q}))$ vanish when $\hbar=\infty$ for all $d>0$, then $A_i=(L_i^{-1}\star)$ by \cite[Proposition 2.10]{IMT} and \cite[Theorem 4.12]{HK24a}. In summary, we have 
 
\begin{prop}\label{quantum-K}
    Suppose that: (i) $D\in \mathbb{Q}[\hbar,\mathbf{Q}]\left \langle L_i^{-1}\hbar^{Q_i\partial_{Q_i}} \right \rangle$ satisfies \cref{D-module relation}; (ii) for $1\leq i \leq r$,  $\mathrm{Coeff}_{Q^d}\left(L_i^{-1}\hbar^{Q_i\partial_{Q_i}}(J_{X}(\hbar,\mathbf{Q}))\right)|_{\hbar=\infty}=0$ for all $d\in \overline{{\rm NE}}(X)\setminus\{0\}$. Then 
     \begin{align}
          D(L_1^{-1} ,\cdots, L_r^{-1},1, Q_1,\cdots, Q_r)(\mathbf{1})=0\quad\text{in }QK(X).
      \end{align} 
\end{prop}
 }

\section{Quantum Chern character for  $\mathbb{P}^{n}$ }

In this section, as a warm-up, we construct a quantum version of  Chern character for $\mathbb{P}^n$, i.e. a ring homomorphism $QK(\mathbb{P}^n)\xrightarrow{\qch} QH(\mathbb{P}^n)$ that is a lift of $K(\mathbb{P}^n)\xrightarrow{\ch}H^*(\mathbb{P}^n)=H^{ev}(\mathbb{P}^{n})$.
 
Let us consider line bundle class $L:=[\mathcal{O}(1)]\in K(\mathbb{P}^{n})$ and hyperplane class $h:=c_1(L) \in H^{2}(\mathbb{P}^{n})$.  We have  the following ring presentations (see e.g. \cite{BuMi}):
\begin{align*}
    QK(\mathbb{P}^{n})&\cong\mathbb{Q}[L^{-1}]\dbb{Q}/((1-L^{-1})^{n+1}-Q),\qquad
    QH(\mathbb{P}^{n})\cong\mathbb{Q}[h]\dbb{q}/(h^{n+1}-q),
\end{align*}
where $L^{-1}=[\mathcal{O}(-1)]$ is the classical inverse of $L$ in $K(\mathbb{P}^{n})$.

We define a map $\qch:QK(\mathbb{P}^{n})\longrightarrow QH(\mathbb{P}^{n})$ as follows. We firstly set
\begin{align*}
   \qch(L^{-1})=e^{-h}\textrm{ and }\qch(Q)=q\left(\frac{1-e^{-h}}{h}\right)^{n+1},
\end{align*}
which both lie in $QH(\bbP^n)$ by our convention in \cref{subsec: QH}.
Secondly, for any $\alpha\in QK(\mathbb{P}^n)$, the above ring presentation implies that there exists $\Phi_\alpha(x,y)\in\mathbb{Q}[x]\dbb{y}$ such that $\alpha=\Phi_\alpha(L^{-1},Q)$. We then set
\[
\qch(\alpha)=\Phi_\alpha\left(\qch(L^{-1}),\qch(Q)\right).
\]
\begin{thm} \label{thm: Pn}
  The map $\qch$ is independent of the choice of $\Phi_\alpha$'s and is a ring homomorphism. Moreover, its classical limit gives ${\rm ch}$.
\end{thm}
\begin{proof}
    To prove the first statement, by the ring presentation for $QK(\mathbb{P}^n)$, it suffices to show that $\left(1-\qch(L^{-1})\right)^{n+1}=\qch(Q)$, or equivalently, $\left(1-e^{-h}\right)^{n+1}=q\left(\frac{1-e^{-h}}{h}\right)^{n+1}$.

    The equality $1-e^{-x}=x\cdot\frac{1-e^{-x}}{x}$ in $\mathbb{Q}\dbb{x}$ implies that $1-e^{-h}=h\star\frac{1-e^{-h}}{h}$. Thus we have
    \[
    \left(1-e^{-h}\right)^{n+1}=h^{n+1}*\left(\frac{1-e^{-h}}{h}\right)^{n+1}=q\left(\frac{1-e^{-h}}{h}\right)^{n+1}.
    \]
    In the last equality, we use the relation $h^{n+1}=q$ in $QH(\mathbb{P}^{n})$. This proves the first statement.
    
The second statement follows directly from the definition of $\qch$. This finishes the proof of the proposition.
\end{proof}

\begin{remark}\label{rmk:Pnunique}
    The image $\qch(Q)$ is actually uniquely determined by $\qch(L^{-1})=e^{-h}$ and the requirement that $\qch$ is a ring homomorphism. Indeed, that $\qch$ is a ring homomorphism implies that $\left(1-\qch(L^{-1})\right)^{n+1}=\qch(Q)$. Together with the relation $h^{n+1}=q$ in $QH(\mathbb{P}^{n})$, we obtain
    \[
    \qch(Q)* h^{n+1}=q\left(1-e^{-h}\right)^{n+1}.
    \] 
    This equation for $\qch(Q)$ has a solution $\qch(Q)=q\left(\frac{1-e^{-h}}{h}\right)^{n+1}$. Moreover, noting $QH(\mathbb{P}^{n})$ is a $\mathbb{Q}\dbb{q}$-algebra, the relation $h^{n+1}=q$ implies that $h$ is NOT a zero divisor in $QH(\mathbb{P}^{n})$. As a consequence, the equation for $\qch(Q)$ has a unique solution.
\end{remark}

\begin{remark}\label{rmk: explainPn}
    The factor $\left(\frac{1-e^{-h}}{h}\right)^{n+1}$ in $\qch(Q)$ can be interpreted as a quantum version of the inverse of Todd class of $T_{\mathbb{P}^n}$. On the one hand, the factor is equal to $f(h)$ with $f(x)=\left(\frac{1-e^{-x}}{x}\right)^{n+1}$. On the other hand, the Euler exact sequence $$0\to\mathcal{O}\to L^{\oplus(n+1)}\to T\mathbb{P}^n\to 0$$ gives
\[
\Td(T_{\mathbb{P}^n})^{-1}=\left({\rm Td}(L)^{n+1}\right)^{-1}=f_{\mathrm{cl}}(h).
\]
So $\left(\frac{1-e^{-h}}{h}\right)^{n+1}$ is mapped to $\Td(T_{\mathbb{P}^n})^{-1}$ via $QH(\mathbb{P}^{n})\xrightarrow{\textrm{mod }q}H^*(\mathbb{P}^n)$.
\end{remark}

\section{Quantum Chern character for $\Fln$ } \label{subsec: QchFl}

In this section, we construct a quantum version of Chern character  for the incidence variety $\Fln=\{ V_{1} \leq V_{n-1} \leq \mathbb{C}^{n}| \dim V_{1}=1,\dim V_{n-1}=n-1\}$, i.e. a ring homomorphism $QK(\Fln)\xrightarrow{\qch} QH(\Fln)$ that gives a lift of $K(\Fln)\xrightarrow{\ch}H^*(\Fln)=H^{\rm ev}(\Fln)$. 

Let   $L_1:=[(\mathcal{S}_1)^{\vee}], L_2:=[\underline{\mathbb{C}^{n}}/\mathcal{S}_{n-1}]\in K(\Fln) $, where $\mathcal{S}_1$ and $\mathcal{S}_{n-1}$  are the tautological vector subbundles of $\Fln $. Note that $\{h_a:=c_1(L_a)\}_{1\leq a\leq 2}$ form a nef basis of $H^{2}(\Fln, \bbZ)$.  Recall that $q_a$ (resp. $Q_a$), $1\leq a\leq 2$, denote the corresponding   Novikov variables  in the quantum cohomology (resp. $K$-theory) of $\Fln$. We have  

 \begin{prop}[\protect{\cite[Proposition 7.2]{ChPe}}] \label{prop: QHFl}
   The small quantum cohomology ring   of $\Fln$ is canonically given by
$      QH(\Fln) \cong \mathbb{Q}[h_1,h_2]\dbb{q_1,q_2}/(f_{1}^{q},f_{2}^{q}),
$ 
where 
\begin{align*}
  f_1^{q}(h_1,h_2,q_1,q_2)&=h_2^{n}-q_2(h_1+h_2),\\
  f_2^{q}(h_1,h_2,q_1,q_2)&=\sum_{l=0}^{n-1}(-1)^{n-1-l}h_1^{l}h_2^{n-1-l}-q_1-(-1)^{n-1}q_2.   
\end{align*}
\end{prop}
 \noindent The   isomorphism in the above proposition is canonical in the sense  that the hyperplane class $h_a$ in $QH(\Fln)$ is maps to the generator $h_a$ on the right, for $a=1, 2$. Moreover, there is a canonical selfduality of $\Fln$, which induces an automorphism of the small quantum cohomology by interchanging $(h_1, q_1)$ and $(h_2, q_2)$. As an immediate consequence, we have 

 \begin{cor}\label{cor--h1h2QHrelationFl}
    We have $h_a^{n}=q_a(h_1+h_2)$ in $QH(\Fln)$ for $a=1,2$.
\end{cor}

Ring presentations for $QK(\Fln)$ could be obtained from \cite{GM+23, Xu}. For our purpose of constructing ${\rm qch}$, we treat $\Fln$ as the special Milnor hypersurface $H_{n-1, n-1}$, and read off one from the general case $QK(H_{n-1, m-1})$  in \cref{quantum K ring presentation} directly; namely
\begin{prop}
The small quantum $K$-theory of $\Fln$ is canonically given by
\begin{align*}
     QK(\Fln) \cong \mathbb{Q}[L_1^{-1},L_2^{-1}]\dbb{Q_1,Q_2}/(F_1^{Q},F_2^{Q}),
\end{align*}
\begin{align*}
   \mbox{where } \,\,   F_1^{Q}(L_1^{-1},L_2^{-1},Q_1,Q_2)&=(1-L_2^{-1})^{n}-Q_2+Q_2(L_1^{-1}L_2^{-1}), \\
 F_2^{Q}(L_1^{-1},L_2^{-1},Q_1,Q_2)&=F_2(L_1^{-1},L_2^{-1})-Q_2 (L_1^{-1})^{n-1}-(-1)^{n-1}Q_1 L_2^{-1},
\end{align*}
 with 
 $F_2(x,y)=(-1)^{n-1}(1-x)^{n-1}+\sum\limits_{t=1}^{n-1}(-1)^{n-1-t}x^{t-1}(1-x)^{n-t-1}(1-y)^{t}.$ 
\end{prop}

We define a map
$QK(\Fln)\xrightarrow{\qch} QH(\Fln)$ as follows. Firstly, for $a=1,2$, we set
\[\qch(L_{a}^{-1})=e^{-h_a}\quad\textrm{ and }\quad\qch(Q_a)=q_a\left(\frac{1-e^{-h_a}}{h_a}\right)^{n}\ast\left(\frac{(h_1+h_2)}{1-e^{-(h_1+h_2)}}\right).\]
Secondly, for any $\alpha\in QK(\Fln)$, the above ring presentation implies that there exists $\Phi_{\alpha}(x_1,x_2,y_1,y_2)\in\mathbb{Q}[x_1,x_2]\dbb{y_1,y_2}$ such that $\alpha=\Phi_\alpha\left(L_1^{-1},L_2^{-1},Q_1,Q_2\right)$, and we set
\[
\qch(\alpha)=\Phi_\alpha\left(\qch(L_1^{-1}),\qch(L_2^{-1}),\qch(Q_1),\qch(Q_2)\right).
\]

\begin{thm}\label{prop--qchwelldefineFl}
    The map $\qch$ is independent of the choice of $\Phi_\alpha$'s and is a ring homomorphism. Moreover, its classical limit gives $\ch$.
\end{thm}

 To prove the theorem, we prepare  several lemmas first.  
Recall that in a commutative ring $R$, an element $x\in R$ is not a {zero-divisor} if and only if  for any $y\in R$, we have $y=0$ whenever  $xy=0$. We first assume the   next lemma, and leave its proof at the end of the appendix, which will involve  mirror symmetry for $\Fln$.
 
\begin{lemma}\label{lemma--h1+h2nonzerodivisorQHFl}
    The element $h_1+h_2$ is  not a zero-divisor in $QH(\Fln)$.
\end{lemma}

\begin{lemma}\label{lemma--h1andh1+h2arenonzerodivisorFl}
    The element $1-e^{-(h_1+h_2)}$ is not a zero-divisor in $QH(\Fln)$.
\end{lemma}
\begin{proof}
     The equality $(1-e^{-x})\cdot\frac{x}{1-e^{-x}}=x$ in $\mathbb{Q}\dbb{x}$ implies that $\left(1-e^{-(h_1+h_2)}\right)\ast\frac{(h_1+h_2)}{1-e^{-(h_1+h_2)}}=h_1+h_2$. So it follows from \cref{lemma--h1+h2nonzerodivisorQHFl} that $1-e^{-(h_1+h_2)}$ is not a zero-divisor.
\end{proof}

\begin{lemma}\label{lemma--hasimplifyqchQaQHFl}
 For $a\in\{1,2\}$, we have $(1-e^{-(h_1+h_2)})\ast\qch(Q_a)=(1-e^{-h_a})^{n}$ in $QH(\Fln).$ 
\end{lemma}
\begin{proof}
   By the equalities 
    \(
    x\cdot\frac{1-e^{-x}}{x}=1-e^{-x}\textrm{ and }(1-e^{-x})\cdot\frac{x}{1-e^{-x}}=x
    \)
    in $\mathbb{Q}\dbb{x}$, we have
    \[
    h_a\ast\frac{1-e^{-h_a}}{h_a}=1-e^{-h_a}\textrm{ and }(1-e^{-(h_1+h_2)})\ast\frac{h_1+h_2}{1-e^{-(h_1+h_2)}}=h_1+h_2.
    \]
    So we see that
    \begin{align*}
        (1-e^{-(h_1+h_2)})\ast\qch(Q_a)&=q_a\left(\frac{1-e^{-h_a}}{h_a}\right)^{n}\ast\left((1-e^{-(h_1+h_2)})\ast\frac{h_1+h_2}{1-e^{-(h_1+h_2)}}\right)\\
        &=q_a\left(\frac{1-e^{-h_a}}{h_a}\right)^{n}\ast(h_1+h_2)\\
&=h_a^{n}\ast\left(\frac{1-e^{-h_a}}{h_a}\right)^{n}\qquad(\text{from }\cref{cor--h1h2QHrelationFl})\\
&=(1-e^{-h_a})^n.
    \end{align*}
\end{proof}
\bigskip

\begin{proof}[Proof of \cref{prop--qchwelldefineFl}]
To prove the first statement, by the ring presentation in \cref{prop: QHFl}, it suffices to show that the two elements 
\[
{F}_a^Q\left(\qch(L_1^{-1}),\qch(L_2^{-1}),\qch(Q_1),\qch(Q_2)\right)\in QH(\Fln),\quad a=1,2,
\]
are vanishing.

For ${F}_1^Q$, by \cref{lemma--h1andh1+h2arenonzerodivisorFl}, we only need to prove that 
    \[
(1-e^{-(h_1+h_2)})\ast\left((1-e^{-h_2})^{n}-\qch(Q_2)+\qch(Q_2)\ast e^{-h_1}\ast e^{-h_2}\right)=0.
    \]
    It follows from \cref{lemma--hasimplifyqchQaQHFl} that
    \begin{align*}
        &(1-e^{-(h_1+h_2)})\ast\left((1-e^{-h_2})^{ n}-\qch(Q_2)+\qch(Q_2)\ast e^{-h_1}\ast e^{-h_2}\right)\\
        ={}&(1-e^{-h_2})^{ n}\ast\left((1-e^{-(h_1+h_2)})-1+e^{-h_1}\ast e^{-h_2}\right).
    \end{align*}
    So the required vanishing result follows from $e^{-(h_1+h_2)}=e^{-h_1}\ast e^{-h_2}$.

    For ${F}_2^Q$, by \cref{lemma--h1andh1+h2arenonzerodivisorFl}, we only need to prove that $\mathrm{LHS}=\mathrm{RHS}$, where
    \begin{align*}
\mathrm{LHS}&:=(1-e^{-(h_1+h_2)})\ast F_2(e^{-h_1},e^{-h_2}),\\
    \mathrm{RHS}&:=(1-e^{-(h_1+h_2)})\ast\left(\qch(Q_2)\ast(e^{-h_1})^{\ast(n-1)}+(-1)^{n-1}\qch(Q_1)\ast e^{-h_2}\right).
    \end{align*}
    We use \cref{lemma--hasimplifyqchQaQHFl} to find that 
    \begin{align*}
        \mathrm{RHS}&=(1-e^{-h_2})^{n}\ast(e^{-h_1})^{n-1}+(-1)^{n-1}(1-e^{-h_1})^{n}\ast e^{-h_2}.
    \end{align*}
    For $\mathrm{LHS}$, observe that $F_2(e^{-h_1},e^{-h_2})$ is equal to
    \begin{align*}
        (-1)^{n-1}(1-e^{-h_1})^{n-1}+(1-e^{-h_2})\ast\sum_{l=0}^{n-2}(e^{-h_1}-1)^{n-2-l}\ast\left(e^{-h_1}\ast(1-e^{-h_2})\right)^{l}.
    \end{align*}
    Note that $(e^{-h_1}-1)-\left(e^{-h_1}\ast(1-e^{-h_2})\right)=-(1-e^{-(h_1+h_2)})$, and we find that 
    \begin{align*}
       \mathrm{LHS}={}&(-1)^{n-1}(1-e^{-h_1})^{n-1}\ast(1-e^{-(h_1+h_2)})\\
       &\quad-(1-e^{-h_2})\ast\left((e^{-h_1}-1)^{\ast(n-1)}-(e^{-h_1}\ast(1-e^{-h_2}))^{n-1}\right).
    \end{align*}
    Now one can check that $\mathrm{LHS}=\mathrm{RHS}$, and this proves the first statement. 

    The second statement     directly follows from the definition of $\qch$. This finishes the proof.
\end{proof}

\begin{remark}
  If we assume a prior that  $\qch$ is a ring homomorphism, then the images $\qch(Q_a)$'s are   uniquely determined by $\qch(L_a^{-1})=e^{-h_a}$,  similar to \cref{rmk:Pnunique}.
\end{remark}

\begin{remark}
 Let $\pi_1,\pi_2$ be  the natural projection   from $\Fln$ to $\mathbb{F}\ell_{1,n}, \mathbb{F}\ell_{n-1;n}$, respectively. We have the exact sequences
  $$0\to\mathcal{O}_{\mathbb{F}\ell_{1,n-1;n}}\to L_1^{\oplus n}\to \pi_1^{*}T_{\mathbb{F}\ell_{1,n}}\to 0, 
  \qquad 0\to\mathcal{O}_{\mathbb{F}\ell_{1,n-1;n}}\to L_2^{\oplus n}\to \pi_2^{*}T_{\mathbb{F}\ell_{n-1;n}}\to 0.$$
Similar to \cref{rmk: explainPn}, the factor $\left(\frac{1-e^{-h_a}}{h_a}\right)^{n} $  in $\qch(Q_a)$, $1\leq a\leq 2$, may be interpreted as a quantum version of the inverse of the Todd classes of $\pi_1^{*}T_{\mathbb{F}\ell_{1;n}}$ and $\pi_2^{*}T_{\mathbb{F}\ell_{n-1;n}}$ respectively.
  The factor    $\frac{(h_1+h_2)}{1-e^{-(h_1+h_2)}}$ may be interpreted as a quantum version of the   Todd class  of $L_1\otimes L_2$.
     
\end{remark}

 \section{Quantum $K$-theory of Milnor hypersurfaces}
In the section, we provide a ring presentation of the small quantum $K$-theory of smooth Milnor hypersurfaces.

\subsection{Milnor hypersurfaces} \label{sub: Milnor}
Let $n\geq m\ge 3$ \footnote{We exclude the case $m=2$, for which there is a nontrivial mirror map in the $K$-theoretic $J$-function by quantum Lefschetz principle and \cref{quantum-K} is not applicable directly.}. The (smooth) Milnor hypersurface $H_{n-1,m-1}$ is a   degree-$(1,1)$ hypersurface in $\mathbb{P}^{n-1} \times \mathbb{P}^{m-1}$, defined by the equation $$x_1y_1+x_2y_2+\cdots+x_{m}y_{m}=0, $$
where $[x_1:   \cdots : x_{n}]$ and $[y_1: \cdots: y_{m}]$ are homogeneous coordinates of $\mathbb{P}^{n-1}, \mathbb{P}^{m-1}$, respectively. 
Note that the incidence variety  $\Fln$ is   a degree $(1, 1)$-hypersurface of 
 $\mathbb{P}^{n-1}\times \mathbb{P}^{n-1}$ defined by 
  $\sum_{j=1}^n(-1)^j x_jy_j=0$ via the    Pl$\ddot{u}$cker embedding.
 We can identify  $H_{n-1, n-1}$ with $\Fln$ by a simple coordinate change $(y_j\mapsto (-1)^jy_j)$.

\begin{remark}
    Any smooth Schubert variety in $\Fln$ is either a Milnor hypersurface or a product of projective spaces. 
\end{remark} 
 
Consider  the following commutative diagram:
\begin{equation} \label{commutative}
    \centering
   \xymatrix{
    & H_{n-1,m-1}   \ar@{^(->}[r]^{ \iota \qquad \quad }    \ar[d]^{\pi^{m}_2}  
    &   H_{n-1,n-1}\cong \Fln  \ar[d]^{\pi_{2}}  \\
    &  \mathbb{P}^{m-1}  \ar@{^(->}[r]^{\iota_m \qquad}   & \mathbb{P}^{n-1}\cong \mathbb{F}\ell_{n-1;n} }
\end{equation}
Here $\iota$ is the natural embedding, $\pi_2$ (resp. $\pi_2^{m}$) is the natural projection  to the second factor of the product of projective spaces, and $\iota_m$ denotes the natural inclusion.  
Let  $\mathcal{U}_{n-1}$ be the tautological bundle of $\mathbb{F}\ell_{n-1;n}$. 
   Then $\Fln=\mathbb{P}(\mathcal{U}_{n-1})$ is a $\mathbb{P}^{n-2}$-bundle over $\mathbb{P}^{n-1}$, and 
      $H_{n-1,m-1}=  \mathbb{P}^{m-1} \times_{\mathbb{F}\ell_{n-1;n}} \Fln= \mathbb{P}(\iota_{m}^{*}\mathcal{U}_{n-1})$ is a $\mathbb{P}^{n-2}$-bundle over $\mathbb{P}^{m-1}$.

\subsection{Ring presentation of  $K(H_{n-1,m-1})$  }
There is a standard procedure to  obtain a ring presentation of the classical $K$-theory of the projective bundle $H_{n-1,m-1}= \mathbb{P}(\iota_{m}^{*}\mathcal{U}_{n-1})$. Here we provide a precise presentation of  $K(H_{n-1,m-1})$ that is helpful for us to derive a ring presentation of 
 $QK(H_{n-1,m-1})$.

\begin{lemma}  \label{lemma-K-relation}
 For any nonnegative  integers $b, n, t$ with $0\leq b\leq n-t-1$, we have  
\begin{align*}
    \sum_{c=0}^{b}\binom{n}{n-b+c}\binom{t+c}{c}(-1)^{c}=\binom{n-t-1}{b}.
\end{align*}
\end{lemma}
 \begin{proof}
It follows from  differentiating   $\sum_{c=0}^{\infty}z^{c}=\frac{1}{1-z}$, where $(|z|<1)$, that 
 $$A(z):=\sum_{c=0}^{\infty}\binom{t+c}{c}z^{c}=\frac{1}{(1-z)^{t+1}}. $$   The coefficient of $z^{b}$-term in $(1+z)^{n}\cdot A(-z)$ is given by  $\sum_{c=0}^{b}\binom{n}{n-b+c}\binom{t+c}{c}(-1)^{c}$. Hence, it is equal to   $\binom{n-t-1}{b}$, by noting  $(1+z)^{n-t-1}=(1+z)^{n}\cdot A(-z)$.
 \end{proof}
Recall the following polynomials defined in the introduction.  
 \begin{align}
      F_1(x,y)&:=(1-y)^{m},\\ 
      F_2(x,y)&:=(-1)^{n-1}(1-x)^{n-1}+\sum_{t=1}^{m-1}(-1)^{n-1-t}x^{t-1}(1-x)^{n-t-1}(1-y)^{t}.\label{expF2}
  \end{align} 
  
\begin{lemma} \label{clasK}
There exists $a(x, y)\in \bbQ[x, y]$ such that 
       $$F_2(x,y)-a(x, y)F_1(x, y)=(-1)^{n}\sum_{l=0}^{n-1} (-x)^{l} \sum_{s=1}^{l+1}\binom{n}{n-1-l+s} (-y)^{s}.$$
 \end{lemma}
\begin{proof}
 Let RHS denote the right-hand side and  $a_1(x, y):= -x^{n-1}(1-y)^{n-m}$. By direct calculation, we have 
  $\mbox{RHS}=a_1(x, y)F_1(x, y)+x^{n-1}+(-1)^{n}\sum_{l=0}^{n-2} (-x)^{l} \sum_{s=1}^{l+1}\binom{n}{n-1-l+s} (-y)^{s}$.
  Note $ (-y)^s=  \sum_{t=1}^{s}(1-y)^{t}\binom{s}{t} (-1)^{s-t}+(-1)^{s}. $ By changing the order of summation over $t$ and $s,l$,  we have 
     $\mbox{RHS}=a_1(x, y)F_1(x, y)+g_0(x)+ \sum_{t=1}^{n-1} g_{t}(x) (1-y)^{t}, $
    where $$g_0(x)=x^{n-1}+(-1)^{n}\sum_{l=0}^{n-2}(-x)^{l} \sum_{s=1}^{l+1}\binom{n}{n-1-l+s}(-1)^{s},$$ $$g_{t}(x)=(-1)^{n}\sum_{l=t-1}^{n-2}(-x)^{l} \sum_{s=t}^{l+1} \binom{n}{n-1-l+s} \binom{s}{t} (-1)^{s-t}. $$
   By Lemma $\ref{lemma-K-relation}$ with $t=0$, we have 
   $$g_{0}(x)=x^{n-1}+(-1)^{n}\sum_{l=0}^{n-2}(-x)^{l} (-1)\sum_{c=0}^{l}\binom{n}{n-l+c}(-1)^{c}=x^{n-1}+(-1)^{n}\sum_{l=0}^{n-2}(-x)^{l} (-1) \binom{n-1}{l}.$$
   Hence, $g_0(x)=(-1)^{n-1}(1-x)^{n-1}$.
   For $1\leq t\leq m-1$,  by substituting   $l$ with $b=l+1-t$ and then applying  Lemma \ref{lemma-K-relation} with $c=s-t$, we have
    $$g_{t}(x)=(-1)^{n}\sum_{b=0}^{n-t-1}(-x)^{b+t-1} \sum_{s=t}^{b+t} \binom{n}{n-b-t+s}\binom{s}{t} (-1)^{s-t}=(-1)^{n}\sum_{b=0}^{n-t-1}\binom{n-t-1}{b}(-x)^{b+t-1}.$$
 Hence, $g_t(x)=(-1)^{n-1-t}x^{t-1}(1-x)^{n-t-1}$. 
 
 Set $a(x, y)=a_1(x, y)+ \sum_{t=m}^{n-1}g_t(x)(1-y)^{t-m}$. We are done. 
\end{proof}
Recall that  $\mathcal{S}_1$ 
(resp. $\mathcal{S}_{n-1}$) denotes the tautological vector subbundle  of $\Fln $ whose fiber at $V_\bullet$ is given by the vector space $V_1$ (resp. $V_{n-1}$). Let $P_1$ (resp. $P_2$) be the class of  line  bundle  $\iota^{*}\mathcal{S}_1$ (resp. $\iota^{*}(\underline{\mathbb{C}^{n}}/\mathcal{S}_{n-1})^{\vee})$  in $K(H_{n-1,m-1})$, namely $P_i=L_i^{-1}$ in \cref{mainthmFln}.
\begin{prop}  \label{Classical K}
The classical $K$-theory of $H_{n-1,m-1}$ is generated by $P_1, P_2$ with  
      $$K(H_{n-1,m-1})=\mathbb{Q}[P_1,P_2]/(F_1(P_1,P_2),F_2(P_1,P_2)).$$

\end{prop}
\begin{proof}
Note $H_{n-1,m-1}\cong \mathbb{P}(\iota_{m}^{*}\mathcal{U}_{n-1})$.  By the classical $K$-theory for projective bundles (see e.g. \cite[Theorem 2.16]{Kar}),
we have 
$$K(H_{n-1,m-1})=K(\bbP^{m-1})[P_1]/(P_1^{n-1}-[\lambda^{1}(\iota_{m}^{*}\mathcal{U}_{n-1})]P_1^{n-2}+\cdots+(-1)^{n-1}[\lambda^{n-1}(\iota_{m}^{*}\mathcal{U}_{n-1})]).$$
 Here $\lambda^{k}(\iota_{m}^{*}\mathcal{U}_{n-1})$ is the $k$-th exterior power of $\iota_{m}^{*}\mathcal{U}_{n-1}$, and $K(H_{n-1,m-1})$ is a free $K(\mathbb{P}^{m-1})$-module via the ring homomorphism $(\pi_2^{m})^{*}: K(\mathbb{P}^{m-1})\longrightarrow K(H_{n-1,m-1})$.

 Let $\underline{\bbC^n}$ also denote the trivial bundle over  $\mathbb{F}\ell_{n-1;n}$ by abuse of notation, and denote $\mathcal{Q}:=\underline{\bbC^n}/\mathcal{U}_{n-1}$, which is a line bundle. So $(1-(\iota_{m}^{*}\mathcal{Q})^{\vee})^{m}=0$ in $K(\mathbb{P}^{m-1}).$ 
  Notice $P_2=[\iota^{*}(\underline{\mathbb{C}^{n}}/\mathcal{S}_{n-1})^{\vee})]=[(\pi_2^{m})^{*}(\iota_{m}^{*}\mathcal{Q})^{\vee}] \in K(H_{n-1,m-1})$. Hence,     $(1-P_2)^{m}=0$.
  
 Let $\mathbf{1}\in K(\bbP^{m-1})$ be the class of the trivial line bundle. For   $ 1\leq k \leq n$, it follows from the exact sequence $0\to \mathcal{U}_{n-1}\to \underline{\bbC^{n}}\to \mathcal{Q}\to 0$, together with $\mathcal{Q}$ being a line bundle,  that 
  \begin{align} \label{qbundle}
      \binom{n}{k} \mathbf{1}=\sum_{t+b=k}[\iota_{m}^{*}(\lambda^{t}(\mathcal{U}_{n-1}))]\cdot[\iota_{m}^{*}(\lambda^{b}\mathcal{Q})]=[\lambda^{k-1}(\iota_{m}^{*}\mathcal{U}_{n-1}) \otimes \iota_{m}^{*}\mathcal{Q})]+  [\lambda^{k}(\iota_{m}^{*}\mathcal{U}_{n-1})].
  \end{align}
  In particular, we have $[\lambda^{n-1}(\iota_{m}^{*}\mathcal{U}_{n-1}) \otimes \iota_{m}^{*}\mathcal{Q})]=\mathbf{1}$, implying $[\lambda^{n-1}(\iota_{m}^{*}\mathcal{U}_{n-1})]=[\iota_{m}^{*}\mathcal{Q})^\vee]$. By (reverse) induction on $k$, we have   $[\lambda^{k}(\iota_{m}^{*}\mathcal{U}_{n-1})]=\sum_{s=1}^{n-k}(-1)^{s-1}\binom{n}{k+s} [(\iota_{m}^{*}\mathcal{Q})^{\vee}]^{s}$ in $K(\mathbb{P}^{m-1})$ for $1\leq k\leq n-1$. Hence,  $K(H_{n-1, m-1})$  is   the quotient of $\bbQ[P_1, P_2]$ by the ideal generated by  $(1-P_2)^m$ and 
   \begin{align} \label{classeqn}
       \sum_{l=0}^{n-1}(-1)^{n-1-l} (P_1)^{l}\bigg(\sum_{s=1}^{l+1}(-1)^{s-1}\binom{n}{n-1-l+s} P_2^{s}\bigg).
   \end{align}
 Then we are done by  Lemma \ref{clasK}.
 \end{proof}

\subsection{Ring presentation of $QK(H_{n-1,m-1})$}
 Recall  $L_i=P_i^{-1}\in K(H_{n-1,m-1})$, and   $\{h_i=c_1(L_i)\}_{1\leq i\leq 2}$ form a nef basis of $H^{2}(H_{n-1,m-1}, \bbZ)$.
 We have the Mori cone $\overline{\rm NE}(H_{n-1, m-1})=\mathbb{Z}_{\geq 0}\beta_1+\mathbb{Z}_{\geq 0}\beta_2\subset H_2(H_{n-1, m-1}, \bbZ)$ with 
 with $\int_{\beta_i}h_{j}=\delta_{i,j}$.  
By \cite[Proposition 1]{Tai} and \cite[Theorem 10]{Lee1999}, the small $K$-theoretic $J$-function of   $\mathbb{P}^{n-1} \times \mathbb{P}^{m-1}$ is given  by
 \begin{align}
     J_{\mathbb{P}^{n-1} \times \mathbb{P}^{m-1}}(\hbar, \mathbf{Q})=  \sum_{d_1\geq 0,\, d_2\geq 0}^{\infty} \frac{Q_1^{d_1}Q_2^{d_2}}{    \prod_{l=1}^{d_1}(1-\mathcal{L}_1^{-1}\hbar^{l})^{n} \prod_{l=1}^{d_2}(1-\mathcal{L}_2^{-1}\hbar^{l})^{m}}.
 \end{align}
 Here $\mathcal{L}_1$ (resp. $\mathcal{L}_2$)  is  the class  of the line bundle obtained by the pullback of  $\mathcal{O}_{\mathbb{P}^{n}}(1)$ (resp. $\mathcal{O}_{\mathbb{P}^{m}}(1)$), and we still denote by $Q_{1},Q_2$    the corresponding Novikov variables for  $QK(\mathbb{P}^{n-1} \times \mathbb{P}^{m-1})$ by abuse of notation.  
As a degree-$(1, 1)$ hypersurface of  $\mathbb{P}^{n-1} \times \mathbb{P}^{m-1}$, 
by the quantum Lefschetz principle \cite[the last Theorem]{Giv15} (see also \cite[Theorem 3]{Giv20} for general cases), we obtain  
the small $K$-theoretic $J$-function of $H_{n-1,m-1}$, \footnote{The $J$-function of $H_{2,2}=F\ell_{1,2;3}$ was computed in \cite[Example 2.5]{GiLe} up to a scalar by  $(1-\hbar)$.}  
 \begin{align}
     J_{H_{n-1,m-1}}(\hbar, \mathbf{Q})=   \sum_{d_1\geq 0, d_2\geq 0}^{\infty} \frac{Q_1^{d_1}Q_2^{d_2} \prod_{l=1}^{d_1+d_2}(1-L_1^{-1}\cdot L_2^{-1}\hbar^{l}) }{    \prod_{l=1}^{d_1}(1-L_1^{-1}\hbar^{l})^{n} \prod_{l=1}^{d_2}(1-L_2^{-1}\hbar^{l})^{m}} .
 \end{align}

\begin{thm} \label{Dmrelaton}
Denote $\vartheta_k:=1-L_k^{-1} \hbar^{Q_k \partial_{Q_k}}, k=1,2$. For the $\hbar$-difference operators 
\begin{align*}
     D_1&:=\vartheta_2^m,\qquad 
     D_2:= (-1)^{n-1}(\vartheta_1)^{n-1}+\sum_{t=1}^{m-1}(-1)^{n-1-t}(1-\vartheta_1)^{t-1}\vartheta_1^{n-t-1}\vartheta_2^{t},
\end{align*}
we have 
\begin{align}
     \big(D_1-Q_2+ Q_2\hbar (1-\vartheta_1)(1-\vartheta_2)\big)J_{H_{n-1, m-1}}&=0,\\
     \big(D_2-  (-1)^{n-m}Q_2 (1-\vartheta_1)^{m-1} \vartheta_1^{n-m} -(-1)^{n-1}Q_1(1-\vartheta_2) \big) J_{H_{n-1, m-1}}&=0.
\end{align}
\end{thm}

\begin{proof}
We show the statement by  direct computations as follows.  

\begin{align*}
   D_1 (J_{H_{n-1,m-1}}) &=  \bigg( \sum_{d_2=0} \frac{ 
  { (1-L_2^{-1})^{m}} Q_1^{d_1} \prod_{l=1}^{d_1}(1-L_1^{-1}\cdot L_2^{-1}\hbar^{l}) }{     \prod_{l=1}^{d_1}(1-L_1^{-1}\hbar^{l})^{n}}  \\
    &{}\qquad+  \sum_{ d_2 \geq 1}^{\infty} \frac{ 
  { (1-L_2^{-1}\hbar^{d_2})^{m}} Q_1^{d_1}Q_2^{d_2} \prod_{l=1}^{d_1+d_2}(1-L_1^{-1}\cdot L_2^{-1}\hbar^{l}) }{     \prod_{l=1}^{d_1}(1-L_1^{-1}\hbar^{l})^{n} \prod_{l=1}^{d_2}(1-L_2^{-1}\hbar^{l})^{m}}  \bigg) \\
   &=  0+ \sum_{d_2 \geq 1}^{\infty} \frac{ 
    Q_1^{d_1}Q_2^{d_2} \prod_{l=1}^{{d_1+d_2-1}}(1-L_1^{-1}\cdot L_2^{-1}\hbar^{l}) \cdot 1  }{     \prod_{l=1}^{d_1}(1-L_1^{-1}\hbar^{l})^{n} \prod_{l=1}^{{d_2-1}}(1-L_2^{-1}\hbar^{l})^{m}} \\
    &{}\qquad + Q_2 \hbar \sum_{d_2 \geq 1}^{\infty} \frac{ 
    Q_1^{d_1}Q_2^{d_2-1} \prod_{l=1}^{{d_1+d_2-1}}(1-L_1^{-1}\cdot L_2^{-1}\hbar^{l}) (-L_1^{-1} \hbar^{d_1} \cdot L_2^{-1}\hbar^{d_2-1}) }{     \prod_{l=1}^{d_1}(1-L_1^{-1}\hbar^{l})^{n} \prod_{l=1}^{{d_2-1}}(1-L_2^{-1}\hbar^{l})^{m}} \\
    &=Q_2 J_{H_{n-1,m-1}}-Q_2 \hbar (L_1^{-1} \hbar^{Q_1 \partial_{Q_1}})(L_2^{-1} \hbar^{Q_2 \partial_{Q_2}})(J_{H_{n-1,m-1}}).
\end{align*}
Here  the term $0$ in the second equality follows from the relation $(1-L_2^{-1})^{m}=0$ in $K(H_{n-1, m-1})$. The third  equality holds because  $d_2 \geq 1$ is replaced by $d_2-1 \geq 0$, so that the $J$-function  $J_{H_{n-1,m-1}}$  appears again.

To simplify the expressions, we denote $S_i:=(1-L_i^{-1} \hbar^{d_i}), i=1, 2$. We have

\begin{align*}
    &{}\quad D_2(J_{H_{n-1,m-1}}) \\
    &=   \bigg({F_2(L_1^{-1},L_2^{-1})}+  \sum_{0\neq (d_1,d_2)}^{\infty} Q_1^{d_1}Q_2^{d_2} \frac{ 
 {F_2(L_1^{-1} \hbar^{d_1},L_2^{-1} \hbar^{d_2})}  \prod_{l=1}^{d_1+d_2}(1-L_1^{-1}\cdot L_2^{-1}\hbar^{l}) }{     \prod_{l=1}^{d_1}(1-L_1^{-1}\hbar^{l})^{n} \prod_{l=1}^{d_2}(1-L_2^{-1}\hbar^{l})^{m}}  \bigg) \\
  &= 0+\sum_{0\neq (d_1,d_2)}^{\infty} Q_1^{d_1}Q_2^{d_2} \frac{ 
 {F_2(1-S_1,1-S_2)(-S_1S_2+S_1+S_2) } \prod_{l=1}^{{d_1+d_2-1}}(1-L_1^{-1}\cdot L_2^{-1}\hbar^{l}) }{     \prod_{l=1}^{d_1}(1-L_1^{-1}\hbar^{l})^{n} \prod_{l=1}^{d_2}(1-L_2^{-1}\hbar^{l})^{m}}   \\
  &=  \bigg(\sum_{d_1=0,d_2 \geq 1}^{\infty} Q_2^{d_2} \frac{ 
  ({(-1)^{n-m}(1-S_1)^{m-1}S_1^{n-m}S_2^{m}})  \prod_{l=1}^{{d_2-1}}(1-L_1^{-1}\cdot L_2^{-1}\hbar^{l}) }{  \prod_{l=1}^{d_2}(1-L_2^{-1}\hbar^{l})^{m}}   \\
  &{}\quad+\sum_{d_1 \geq 1,d_2=0}^{\infty} Q_1^{d_1} \frac{ 
  ({(-1)^{n-1}S_1^{n}(1-S_2)})  \prod_{l=1}^{{d_1-1}}(1-L_1^{-1}\cdot L_2^{-1}\hbar^{l}) }{     \prod_{l=1}^{d_1}(1-L_1^{-1}\hbar^{l})^{n} }  \\
  &{}\quad+\sum_{d_1\geq 1,d_2 \geq 1}^{\infty} Q_1^{d_1}Q_2^{d_2} \frac{ 
 {((-1)^{n-1}S_1^{n}(1-S_2)+(-1)^{n-m}(1-S_1)^{m-1}S_1^{n-m}S_2^{m})}\prod_{l=1}^{{d_1+d_2-1}}(1-L_1^{-1}\cdot L_2^{-1}\hbar^{l}) }{     \prod_{l=1}^{d_1}(1-L_1^{-1}\hbar^{l})^{n} \prod_{l=1}^{d_2}(1-L_2^{-1}\hbar^{l})^{m}}
  \bigg)  \\
  &=  \bigg(Q_2\sum_{d_2 \geq 1}^{\infty} Q_1^{d_1}Q_2^{{d_2-1}} \frac{ 
  ((-1)^{n-m}(1-S_1)^{m-1}S_1^{n-m}S_2^{m})  \prod_{l=1}^{d_1+{d_2-1}}(1-L_1^{-1}\cdot L_2^{-1}\hbar^{l}) }{     \prod_{l=1}^{d_1}(1-L_1^{-1}\hbar^{l})^{n} \prod_{l=1}^{{d_2-1}}(1-L_2^{-1}\hbar^{l})^{m} S_2^{m} }   \\
  &{}\qquad+Q_1\sum_{d_1 \geq 1}^{\infty} Q_1^{{d_1-1}}Q_2^{d_2} \frac{ 
  ((-1)^{n-1}S_1^{n}(1-S_2))  \prod_{l=1}^{{d_1-1}+d_2}(1-L_1^{-1}\cdot L_2^{-1}\hbar^{l}) }{   S_1^{n}  \prod_{l=1}^{{d_1-1}}(1-L_1^{-1}\hbar^{l})^{n} \prod_{l=1}^{d_2}(1-L_2^{-1}\hbar^{l})^{m}}  
  \bigg)    \\
&=  (-1)^{n-m} Q_2 (L_1^{-1} \hbar^{Q_1 \partial_{Q_1}})^{m-1} (1-L_1^{-1} \hbar^{Q_1 \partial_{Q_1}})^{n-m} (J_{H_{n-1,m-1}})+(-1)^{n-1}Q_1 (L_2^{-1} \hbar^{Q_2 \partial_{Q_2}})(J_{H_{n-1,m-1}}).
\end{align*}
Here  $F_2(x,y)$ is the polynomial defined in \cref{expF2}.  The  term $0$ in the second equality  follows from the relation $F_2(L_1^{-1},L_2^{-1})=0$ in $K(H_{n-1, m-1})$ by \cref{Classical K}. The term $(-S_1S_2+S_1+S_2)$ in the second equality  comes from    $1-L_1^{-1}L_2^{-1}\hbar^{d_1+d_2}. $
  The last  equality holds because  $d_2 \geq 1$ (resp. $d_1 \geq 1$) could  be replaced by $d_2-1 \geq 0$ (resp. $d_1-1 \geq 0$) after canceling out the same term $S_2^{m}$ (resp. $S_1^{n}$), so that the $J$-function  appears again.
\end{proof}

To achieve a ring presentation of $QK(H_{n-1,m-1})$, it remains to apply the following Nakayama-type result  proved in \cite[Proposition A.3]{GMSZ}. 

\begin{prop}[Gu-Mihalcea-Sharpe-Zou] \label{Nakatype}
    Let $A$ be a Noetherian integral domain, and let $\mathfrak{a} \subset A$ be an ideal. Assume that $A$ is complete in the $\mathfrak{a}$-adic topology. Let $M,N$ be finitely generated $A$-modules. Assume that the $A$-module $N$, and the $A/\mathfrak{a}$-module $N/\mathfrak{a} N$, are both free modules of the same rank $r \leq \infty$, and that we are given an $A$-module homomorphism $f:M \rightarrow N$ such that the induced $A/\mathfrak{a}$-module map $\overline{f}: M/\mathfrak{a}M \rightarrow N/\mathfrak{a}N$ is an isomorphism of $A/\mathfrak{a}$-modules. Then $f$ is an isomorphism.
\end{prop}
Now we restate \cref{thm: presentationQK} as follows.
 \begin{thm}  \label{quantum K ring presentation} Let $n\geq m\geq 3$.
     The small quantum $K$-theory  of    $H_{n-1, m-1}$  is presented by
    \begin{align*}
     QK(H_{n-1, m-1}) \cong \mathbb{Q}[L_1^{-1},L_2^{-1}]\dbb{Q_1,Q_2}/(F_1^{Q}, F_2^{Q}),
\end{align*}
 where  $F_1^{Q}(L_1^{-1},L_2^{-1},Q_1,Q_2)=F_1(L_1^{-1},L_2^{-1})-Q_2+Q_2 L_1^{-1}L_2^{-1}$ and
 $$F_2^{Q}(L_1^{-1},L_2^{-1},Q_1,Q_2) =F_2(L_1^{-1},L_2^{-1})-(-1)^{n-m}Q_2 (L_1^{-1})^{m-1}(1-L_1^{-1})^{n-m}-(-1)^{n-1}Q_1 L_2^{-1}.
 $$  
\end{thm}

\begin{proof}
   Let $M$ be the right hand side of the expected isomorphism,  set  $N:=QK(H_{n-1,m-1})$, and  take the ideal  $\mathfrak{a}=(Q_1,Q_2)$  of 
   $A:=\mathbb{Q}\dbb{Q_1,Q_2}$. Clearly, we have $A/\mathfrak{a}=\mathbb{Q}$. Moreover, $N$ (resp. $N/\mathfrak{a}N$) is a free $A$-module (resp. $A/\frak{a}$-module) of rank $\dim_{\mathbb{Q}}K(H_{n-1,m-1})=m(n-1)<\infty$. Note $n\geq m\geq 3$. It follows directly from the expression of $J_{H_{n-1, m-1}}$ that  each $(d_1, d_2)$-term of  $L_i^{-1}\hbar^{Q_i\partial_{Q_i}}(J_{H_{n-1, m-1}})$   vanishes at $\hbar=\infty$ for $1\leq i \leq 2$, whenever $d_1>0$ or $d_2>0$.
   By     \cref{quantum-K}  and \cref{Dmrelaton},
    the relations $F_1^{Q}(L_1^{-1},L_2^{-1},Q_1,Q_2)=0$ and $F_2^{Q}(L_1^{-1},L_2^{-1},Q_1,Q_2)=0$ hold in $QK(H_{n-1,m-1})$ (with respect to the quantum $K$-product). 
    Therefore we have a canonical ring homomorphism $f: M \rightarrow N$ defined  by $L_i^{-1} \mapsto L_i^{-1}$ and $Q_i \mapsto Q_i, i=1,2$. Since 
    $f(\mathfrak{a}M) \subset \mathfrak{a}N$, it induces an $A/\mathfrak{a}$-module homomorphism $\overline{f}:  M/\mathfrak{a}M \rightarrow N/\mathfrak{a}N$. Note    $M/\mathfrak{a}M \cong (A/\mathfrak{a})\otimes_{A} M \cong \mathbb{Q}[L_1^{-1},L_2^{-1}]/(F_1(L_1^{-1},L_2^{-1}),F_2(L_1^{-1},L_2^{-1})$ and $N/\mathfrak{a}N \cong K(H_{n-1,m-1})$. Thus 
     $\overline{f}$ is an isomorphism of $A/\mathfrak{a}$-modules  by \cref{Classical K}.  Therefore $f$ is an isomorphism of $A$-modules by \cref{Nakatype}.
\end{proof}

\begin{remark}
    The presentation for the special case   $QK(H_{n-1, n-1})$ can also be derived from  the  quantum Littlewood-Richardson rule for  $QK(\Fln)$ \cite[Section 5]{Xu}. Therein, the  Schubert class $\mathcal{O}^{[1,n-1]}$ (resp. $\mathcal{O}^{[2,n]}$)  is given by  $1-L_1^{-1}$, (resp. $1-L_2^{-1}$).    
\end{remark}

\section{Appendix: proof of \cref{lemma--h1+h2nonzerodivisorQHFl}}

By \cite{BCFKS}, the toric superpotential mirror to  $\Fln$ is the   Laurent polynomial
\[f_{\rm tor}:=\sum_{k=1}^{2n-3} \frac{x_{k}}{x_{k-1}}+\frac{q_2}{x_{n-2}}+\frac{x_{n}}{q_2}+\frac{q_1 q_2}{x_{2n-3}}\in\mathbb{Q}[x_{1}^{\pm 1},\cdots, x_{2n-3}^{\pm 1},q_1^{\pm 1},q_2^{\pm 1} ],
\]
where $x_0:=1$. The Jacobi ring of $f_{\rm tor}$ is defined by 
\begin{align} \label{Jideal}
    {\rm Jac}(f_{\rm tor}):=\mathbb{Q}[x_{1}^{\pm 1},\cdots, x_{2n-3}^{\pm 1},q_1^{\pm 1},q_2^{\pm 1} ]\Big/\Big(x_1 \frac{\partial f_{\rm tor}}{\partial x_1},\cdots,x_{2n-3}\frac{\partial f_{\rm tor}}{\partial x_{2n-3}}\Big).
\end{align}

\begin{prop} \label{MS}
There is an injective homomorphism of $\mathbb{Q}\dbb{q_1,q_2}$-algebras
   \[
   \Phi: QH(\Fln)\hookrightarrow {\rm Jac}(f_{\rm tor})\otimes_{\mathbb{Q}[q_1,q_2]}\mathbb{Q}\dbb{q_1,q_2},
   \]
   such that  
   \(
   \Phi(h_1)=[\frac{q_1q_2}{x_{2n-3}}],\quad\Phi(h_2)=[x_1].
   \)
 
\end{prop}
\begin{proof}
Assume   $n \geq 4$ first.  
Let $R_a$ be the relation given by $x_a\frac{\partial f_{\rm tor}}{\partial x_a}$. From $R_1,...,R_{n-3}$ and $R_{n+1}, \ldots, R_{2n-3}$, we get 
 \begin{align}\label{R1}
    x_k=\begin{cases}
     x_1^k,& \mbox{if }2\le k\le n-2,\\
     (\frac{x_{2n-3}}{q_1 q_2})^{2n-3-k}x_{2n-3},&
  \mbox{if } n\leq k \leq 2n-4.
   \end{cases}
 \end{align} 
\noindent It follows immediately that 
\begin{align}\label{Jac1}
    {\rm Jac}(f_{\rm tor})\cong \mathbb{Q}[x_{1}^{\pm 1},  x_{2n-3}^{\pm 1}, x_{n-1}^{\pm 1}, q_1^{\pm 1},q_2^{\pm 1} ]\big/\big(R_{n-2}, R_{n-1}, R_{n}\big).
\end{align}
It follows from $R_{n-2}$ that 
\begin{align}\label{Rn-1}
    x_{n-1}=x_1^{n-1}{-}q_2.
\end{align}
It follows from $R_{n-1}$ and \eqref{R1}  that in ${\rm Jac}(f_{\rm tor})$, 
\begin{align}\label{Rn}
     x_n=\frac{(x_{1}^{n-1}-q_2)^{2}}{x_{1}^{n-2}}=  (\frac{x_{2n-3}}{q_1 q_2})^{n-3}x_{2n-3}.
\end{align}
 By $R_n$ we have $\frac{x_n}{x_{n-1}}=\frac{x_{n+1}}{x_n}-\frac{x_n}{q_2}$. By
 \eqref{Rn-1} and \eqref{Rn}, we have $\frac{x_n}{x_{n-1}}=\frac{x_1^{n-1}-q_2}{x_1^{n-2}}$. By \eqref{R1}, we have  $\frac{x_{n+1}}{x_n}=\frac{q_1q_2}{x_{2n-3}}$. By \eqref{Rn}, we have  $\frac{x_n}{q_2}=\frac{1}{q_2} \frac{(x_{1}^{n-1}-q_2)^{2}}{x_{1}^{n-2}}$. All these relations  together  give the relation in ${\rm Jac}(f_{\rm tor})$ 
\begin{align}\label{Rn+1}
  \frac{x_{1}^{n }}{q_2}-x_1-\frac{q_1 q_2}{x_{2n-3}}=0. 
\end{align}
One can check that   the ideal  $(R_{n-2}, R_{n-1}, R_n)$ in \eqref{Jac1} is generated by the three relations \eqref{Rn-1}, \eqref{Rn}, \eqref{Rn+1}. 
Recall $QH(\Fln)=\bbQ[h_1, h_2]\dbb{q_1, q_2}/(f_1^q(h_1, h_2, q_1, q_2), f_2^q(h_1, h_2, q_1, q_2))$ by \cref{prop: QHFl}.
It remains to show the relations \eqref{Rn} and \eqref{Rn+1} are equivalent 
to   $f_1^q({q_1q_2\over x_{2n-3}}, x_1, q_1, q_2)$ and $f_2^q({q_1q_2\over x_{2n-3}}, x_1, q_1, q_2)$ in \eqref{Jac1}.

 Since $q_2$ is a unit in \eqref{Jac1}, \eqref{Rn+1} is equivalent to $x_1^n-q_2(x_1+{q_1q_2\over x_{2n-3}})= f_1^q({q_1q_2\over x_{2n-3}}, x_1, q_1, q_2)$.
By direct calculations using \eqref{Rn} and \eqref{Rn+1}, we have
\begin{align}\label{Rn+1new}
    q_1q_2-(\frac{q_1 q_2}{x_{2n-3}})^{n-2}(-\frac{x_{1}^{n-1}-q_2}{x_{1}^{n-2}}+\frac{q_1 q_2}{x_{2n-3}}) q_2=0
\end{align}
 Therefore, in \eqref{Jac1}, we note $q_2$ is a unit again, and then  by \eqref{Rn+1new} we have
\begin{align*}
   0= &q_1-(\frac{q_1 q_2}{x_{2n-3}})^{n-2}(-x_1+\frac{q_2}{x_{1}^{n-2}}+\frac{q_1 q_2}{x_{2n-3}}) \\
    ={}&q_1-\big((\frac{q_1 q_2}{x_{2n-3}})^{n-1}- (\frac{q_1 q_2}{x_{2n-3}})^{n-2} x_1+\frac{q_1 q_2}{x_{2n-3}}\frac{q_2}{x_{1}^{n-2}}(\frac{q_1 q_2}{x_{2n-3}})^{n-3}  \big) \\
    ={}&q_1-\big((\frac{q_1 q_2}{x_{2n-3}})^{n-1}- (\frac{q_1 q_2}{x_{2n-3}})^{n-2} x_1+(\frac{q_1 q_2}{x_{2n-3}})^{n-3} x_1^{2}-\frac{q_2}{x_1^{n-3}}(\frac{q_1 q_2}{x_{2n-3}})^{n-3 } \big) \\
    ={}&q_1-\big((\frac{q_1 q_2}{x_{2n-3}})^{n-1}- (\frac{q_1 q_2}{x_{2n-3}})^{n-2} x_1+(\frac{q_1 q_2}{x_{2n-3}})^{n-3} x_1^{2}\\
    &\qquad\qquad\qquad\qquad-(\frac{q_1 q_2}{x_{2n-3}})^{n-4} x_1^{3}+\cdots+(-1)^{n-1}x_1^{n-1}-(-1)^{n-1} q_2\big). 
\end{align*}
 Here the third equality follows from \eqref{Rn+1}, and the last equality follows from \eqref{Rn+1} together with the induction on the power of $x_1$. Therefore  we obtain the following relation
\begin{align}
    0=\sum_{k=0}^{n-1}(-1)^{k}(\frac{q_1 q_2}{x_{2n-3}})^{n-1-k} x_1^{k}-(-1)^{n-1}q_2-q_1=f_2^q({q_1q_2\over x_{2n-3}}, x_1, q_1, q_2).
\end{align}

It follows directly from the process of obtaining $f_1^q, f_2^q$ that  
$$\big(R_{n-2}, R_{n-1}, R_{n}\big)=\big(x_{n-1}-(x_1^{n-1}-q_2), f_1^q({q_1q_2\over x_{2n-3}}, x_1, q_1, q_2), f_2^q({q_1q_2\over x_{2n-3}}, x_1, q_1, q_2)\big)$$
   as ideals in  $\mathbb{Q}[x_{1}^{\pm 1},  x_{2n-3}^{\pm 1}, x_{n-1}^{\pm 1}, q_1^{\pm 1},q_2^{\pm 1} ]$.
Hence,   $h_1\mapsto {q_1q_2\over x_{2n-3}}, h_2\mapsto x_1$, $q_1\mapsto q_1$ and $q_2\mapsto q_2$ cannonically defines an injective ring homomorphism $\Phi$.

The case $n=3$ can be verified by similar but easier calculations (in which case $\mathbb{F}\ell_{1, 2; 3}$ is a complete flag variety and the expected statement is known earlier to the experts). 
\end{proof}

\begin{remark}
    The injective homomorphism $\Phi$ can naturally be extended to an isomorphism  from a suitable localization of $QH(\Fln)$ to ${\rm Jac}(f_{\rm tor})\otimes_{\mathbb{Q}[q_1,q_2]}\mathbb{Q}\dbb{q_1,q_2}$.
\end{remark}

\bigskip

\begin{proof}[Proof of \cref{lemma--h1+h2nonzerodivisorQHFl}]
By \cref{MS},  $QH(\Fln)$ is a subring of ${\rm Jac}(f_{\rm tor})\otimes_{\mathbb{Q}[q_1,q_2]}\mathbb{Q}\dbb{q_1,q_2}$ via the embeding $\Phi$.  Note that both $\Phi(h_1)={q_1q_2\over x_{2n-3}}$ and $q_1$ are invertible in the Jacobi ring.
Since $h_1^{n}=q_1(h_1+h_2)$ holds in $QH(\Fln)$, $\Phi(h_1+h_2)=q_1^{-1}\Phi(h_1)^n$ is invertible in the Jacobi ring. Hence, 
$h_1+h_2$ is not a zero-divisor in (the subring) $QH(\Fln)$. 
\end{proof}


\begin{thebibliography}{99}
 

  \bibitem[$\mbox{AHK}^+$25]{AHK+25} K. Amini, I. Huq-Kuruvilla, L. C. Mihalcea, D. Orr and  W. Xu, \emph{
Toda-type presentations for the quantum $K$ theory of partial flag varieties}, preprint at arXiv: math.AG/2504.07412 (2025).


  \bibitem[ACTI22]{ACT} D Anderson,  L Chen and H. H. Tseng, \emph{
On the finiteness of quantum K-theory of a homogeneous space},  Int  Math  Res  Not,  2022, 1313--1349. With Appendix B by H. Iritani. 
 \bibitem[AH61]{AtHi} M.F. Atiyah, F.  Hirzebruch, \emph{
Vector bundles and homogeneous spaces}, Proc. Sympos. Pure Math., Vol. III, pp. 7--38, American Mathematical Society, Providence, RI, 1961.

 \bibitem[BeFa97]{BeFa} K. Behrend and B. Fantechi, \emph{The intrinsic normal cone}, Invent. Math. 128 (1997), 45--88.
 \bibitem[BCFKS00]{BCFKS}V. Batyrev, I. Ciocan-Fontanine, B. Kim and  D. van Straten,\,{\it Mirror symmetry and toric degenerations of partial flag manifolds}, Acta Math. 184 (2000), no. 1, 1-39.
 
 \bibitem[BrFi11]{BrFi} A. Braverman and M. Finkelberg, \emph{Semi-infinite Schubert varieties and quantum $K$-theory of flag manifolds},
J. Amer. Math. Soc. 27 (2014), no. 4, 1147--1168.
 \bibitem[BCMP13]{BCMP13}  A. S. Buch, P.-E. Chaput, L. C. Mihalcea and N. Perrin, \emph{Finiteness of cominuscule quantum $K$-theory}, Ann. Sci. \'Ec. Norm. Sup\'er. (4) 46 (2013), no. 3, 477--494.
 \bibitem[BCMP16]{BCMP16}  A. S. Buch, P.-E. Chaput, L. C. Mihalcea and N. Perrin, \emph{Rational connectedness implies finiteness of quantum $K$-theory}, Asian J. Math. 20 (2016), no. 1, 117--122.

 \bibitem[BCMP18]{BCMP18}  A. S. Buch, P.-E. Chaput, L. C. Mihalcea and N. Perrin, \emph{A Chevalley formula for the equivariant quantum $K$-theory of cominuscule varieties}, Algebr. Geom. 5 (2018), no. 5, 568--595.  

\bibitem[BCP23]{BCP} A. S. Buch, P.-E. Chaput and N. Perrin, \,{\it Seidel and Pieri products in cominuscule quantum K-theory}, preprint at arXiv: math.AG/2308.05307 (2023).
\bibitem[BM11]{BuMi}  A. S. Buch and   L. C. Mihalcea, \emph{Quantum $K$-theory of Grassmannians},  Duke Math. J. 156 (2011), no. 3, 501--538. 
 \bibitem[CP11]{ChPe} P.E. Chaput and N. Perrin,
{\em On the quantum cohomology of adjoint varieties},
Proc. Lond. Math. Soc. (3) 103 (2011), no. 2, 294--330.

 \bibitem[Co03]{Coa} T. Coates, \emph{Riemann-Roch theorems in Gromov-Witten theory}, ProQuest LLC, Ann Arbor, MI, 2003, 147 pp.
 \bibitem[CG06]{CoGi}  T. Coates and A. Givental, \emph{ Quantum cobordisms and formal group laws}, Progr. Math., 244, Birkhäuser Boston, Inc., Boston, MA, 2006, 155–171.

 \bibitem[ChLe20]{ChLe}  C. H. Chow and N. C. Leung, \emph{Quantum $K$-theory of $G/P$ and $K$-homology of affine Grassmannian},  preprint at arXiv: math.AG/2201.12951 (2022).

 \bibitem[CoLi20]{CoLi} K. Costello and S. Li, \emph{Anomaly cancellation in the topological string}, Adv. Theor. Math. Phys. 24 (2020), no. 7, 1723--1771.

\bibitem[EH16]{EiHa} D. Eisenbud and J. Harris, \emph{
3264 and all that—a second course in algebraic geometry},
Cambridge University Press, Cambridge, 2016.


 \bibitem[FP97]{FuPa} W. Fulton, R. Pandharipande, {\it Notes on stable maps and quantum cohomology. Algebraic geometry-Santa Cruz}, Proc. Sympos. Pure Math., 62, Part 2, Amer. Math. Soc., Providence, RI, 1997.

\bibitem[FO99]{FuOn}  Kenji Fukaya, Kaoru Ono. Arnold conjecture and Gromov-Witten invariant. Topology 38 (1999), no. 5, 933–1048.

  \bibitem[Gi00]{Giv} A. Givental, \emph{On the WDVV equation in quantum K-theory}, Michigan Math. J. 48 (2000), 295--304.
\bibitem[Gi04]{Giv04} A. Givental, \emph{Symplectic geometry of Frobenius structures}, Aspects Math., E36, Friedr. Vieweg \& Sohn, Wiesbaden, 2004, 91--112.
 
\bibitem[Gi15]{Giv15} A. Givental, \emph{Permutation-equivariant quantum $K$-theory V. Toric $q$-hypergeometric functions}, preprint at arXiv: math.AG/1509.03903 (2015).
\bibitem[Gi20]{Giv20} A. Givental, \emph{Permutation-equivariant quantum $K$-theory X. Quantum Hirzebruch-Riemann-Roch in genus 0},
SIGMA Symmetry Integrability Geom. Methods Appl. 16 (2020), Paper No. 031, 16 pp.

 \bibitem[GL03]{GiLe} A. Givental and  Y.-P. Lee, \emph{Quantum $K$-theory on flag manifolds, finite-difference Toda lattices and quantum groups},
       Invent. Math. 151 (2003), no. 1, 193--219.
 \bibitem[GT14]{GiTo}  A. Givental and V.  Tonita, \emph{The Hirzebruch-Riemann-Roch theorem in true genus-0 quantum $K$-theory}, 
Math. Sci. Res. Inst. Publ., 62, Cambridge University Press, New York, 2014, 43--91.
 \bibitem[Gr58]{Gro} A. Grothendieck, \emph{La th\'eorie des classes de Chern},         Bull. Soc. Math. France 86 (1958), 137--154.

 \bibitem[$\mbox{GM}^+$23]{GM+23} W. Gu, L. C. Mihalcea, E. Sharpe, W. Xu, H. Zhang and H. Zou, \emph{Quantum $K$ Whitney relations for partial flag varieties}, preprint at arXiv: math.AG/2310.03826 (2023).
 \bibitem[$\mbox{GM}^+$24]{GM+24} W. Gu, L. C. Mihalcea, E. Sharpe, W. Xu, H. Zhang and H. Zou, \emph{Quantum $K$ theory rings of partial flag manifolds}, J. Geom. Phys. 198 (2024), Paper No. 105127, 30pp.

\bibitem[GMSZ22]{GMSZ} W. Gu, L. C. Mihalcea, E. Sharpe and H. Zou,  \emph{Quantum $K$ theory of Grassmannians, Wilson line operators, and Schur bundles}, preprint at arXiv: math.AG/2208.01091 (2022).


 
\bibitem[HK24a]{HK24a} I. Huq-Kuruvilla, \emph{Relations in Twisted Quantum $K$-Rings},
         preprint at arXiv: math.AG/2406.00916  (2024).
\bibitem[HK24b]{HK} I. Huq-Kuruvilla, \emph{Quantum $K$-rings of partial flag varieties, Coulomb branches, and the Bethe ansatz},
         preprint at arXiv: math.AG/2409.15575  (2024).
    
\bibitem[IIM20]{IIM} T. Ikeda, S. Iwao and T. Maeno, \emph{Peterson isomorphism in K-theory and relativistic Toda lattice}, Int. Math. Res. Not. IMRN 19 (2020), 6421--6462.

 \bibitem[IMT15]{IMT} H. Iritani, T. Milanov and V. Tonita, \emph{Reconstruction and convergence in quantum  $K$-theory via difference equations},
Int. Math. Res. Not. IMRN 2015, no. 11, 2887--2937.

 \bibitem[Ka78]{Kar}  M. Karoubi, \emph{$K$-theory. An introduction}, Grundlehren der Mathematischen Wissenschaften, Band 226, Springer-Verlag, Berlin-New York, 1978.  

 \bibitem[Ka18]{Kat18} S. Kato, \emph{Loop structure on equivariant $K$-theory of semi-infinite flag manifolds}, to appear in Ann. Math.;  preprint
 at arXiv: math.AG/1805.01718.
 \bibitem[Ka19]{Kat19} S. Kato, \emph{On quantum $K$-groups of partial flag manifolds},  preprint at arXiv: math.AG/1906.09343 (2019).

\bibitem[KPSZ21]{KPSZ} P. Koroteev, P. P. Pushkar, A. V. Smirnov  and A. M. Zeitlin, \emph{Quantum $K$-theory of quiver varieties and many-body systems}, Selecta Math. (N.S.) 27 (2021), no. 5, Paper No. 87, 40pp.  

\bibitem[KN24]{KoNa} T. Kouno and  S. Naito, \emph{Borel-type presentation of the torus-equivariant quantum $K$ ring of flag manifolds of type $C$}, 
preprint at arXiv: math.AG/2410.10575 (2024).

\bibitem[KLNS24]{KLNS} T. Kouno, C. Lenart, S. Naito and D. Sagaki, \emph{Quantum $K$-theory Chevalley
formulas in the parabolic case}, J. Algebra 645 (2024), 1--53. With Appendix B joint with W. Xu.
\bibitem[LLMS18]{LLMS} T. Lam, C. Li, L. C. Mihalcea and M. Shimozono, \emph{A conjectural Peterson isomorphism in $K$-theory}, J. Algebra 513 (2018), 326--343.

\bibitem[LNS24]{LNS} C. Lenart, S. Naito and D. Sagaki, \emph{A general Chevalley formula for semi-infinite flag manifolds and quantum $K$-theory},  
Selecta Math. (N.S.) 30 (2024), no. 3, Paper No. 39, 44 pp.
\bibitem[Lee99]{Lee1999}  Y.-P. Lee, \emph{Quantum $K$-theory},   ProQuest LLC, Ann Arbor, MI, 1999, 58 pp.


 \bibitem[Le04]{Lee}  Y.-P. Lee, \emph{ Quantum K-theory. I. Foundations},  Duke Math. J. 121 (2004), no. 3, 389--424. 
 \bibitem[LL17]{LeLi} N.C. Leung and C. Li, \emph{An update of quantum cohomology of homogeneous varieties}, Proceedings of the Sixth International Congress of Chinese Mathematicians. Vol. II, 211–235, Adv. Lect. Math. (ALM), 37, Int. Press, Somerville, MA, 2017. 

\bibitem[LLSY25]{LLSY} C. Li, Z. Liu, J. Song and M. Yang, \emph{On  Seidel representation in quantum $K$-theory of Grassmannians},   Sci. China Math. 68 (2025), no. 7, 1523–1548.

\bibitem[LT98]{LiTi} J. Li and G. Tian, {\em Virtual moduli cycles and Gromov-Witten invariants of algebraic varieties}, J. Amer. Math. Soc. 11 (1998), no. 1, 119--174.

\bibitem[MNS25a]{MNS25a} T. Maeno, S. Naito, and D. Sagaki, \emph{A presentation of the torus-equivariant quantum $K$-theory ring of flag manifolds of type $A$,
Part I: The defining ideal}, J. Lond. Math. Soc. (2) 111 (2025), no. 3, Paper No. e70095, 43pp.
\bibitem[MNS25b]{MNS25b} T. Maeno, S. Naito, and D. Sagaki, \emph{A presentation of the torus-equivariant quantum $K$-theory ring of flag manifolds of type $A$, Part II: quantum double Grothendieck polynomials}, Forum Math. Sigma 13 (2025), Paper No. e19, 26pp.  
 


\bibitem[RT95]{RuTi} Y. Ruan and G. Tian, \emph{A mathematical theory of quantum cohomology},  J. Differential Geom. 42 (1995), no. 2, 259--367.



  \bibitem[Ta13]{Tai}   K. Taipale, \emph{$K$-theoretic $J$-functions of type $A$ flag varieties}, Int. Math. Res. Not. IMRN 2013, no. 16, 3647--3677.
  
  


\bibitem[Xu24]{Xu}  W. Xu, \emph{Quantum $K$-theory of incidence varieties}, Eur. J. Math. 10 (2024), no. 2, Paper No. 22, 47 pp.


 
\end{thebibliography}
\end{document}